\documentclass[12pt, english, letterpaper]{article}
\usepackage{babel,amsmath, amssymb}
\usepackage{graphicx}
\usepackage{amsmath}
\usepackage{bm}

\newcommand{\Sp}{{\mathbb{S}^2}}
\newcommand{\C}{\mathbb{C}}
\newcommand{\R}{\mathbb{R}}
\newcommand{\Z}{\mathbb{Z}}

\renewcommand{\chi}{{\cal X}}
\newtheorem{proposition}{Proposition}[section]
\newtheorem{corollary}[proposition]{Corollary}
\newtheorem{lemma}[proposition]{Lemma}
\newtheorem{theorem}[proposition]{Theorem}
\newtheorem{notita}[proposition]{Remark}
\newenvironment{remark}{\begin{notita}\rm}{\hfill$\Box$\\[0.5ex]\end{notita}}

\newenvironment{proof}{{\it Proof}. }{\hfill$\Box$\\[0.5ex]}

\newcommand{\x}{\boldsymbol{x}}
\newcommand{\y}{\boldsymbol{y}}
\newcommand{\p}{\boldsymbol{p}}
\newcommand{\q}{\boldsymbol{q}}
\renewcommand{\d}{\boldsymbol{d}}
\newcommand{\ddh}{\widehat{\d}}
\newcommand{\xh}{\widehat{\x}}
\newcommand{\sph}{\mathbb{S}^2}

\author
{V. Dom\'{\i}nguez\thanks{
Departamento Ingenier\'ia Matem\'atica e Inform\'atica,  Universidad P\'ublica
de Navarra, Campus de Tudela, 31500, Tudela, Spain.
\newline {\tt (victor.dominguez@unavarra.es)}}
\and
M. Ganesh\thanks{Department of Applied Mathematics \& Statistics,
Colorado School of Mines,
Golden, CO 80401. ~~~~~~~~~~~~{\tt (mganesh@mines.edu)}}
}

\title{Interpolation and cubature approximations and analysis 
for a class of  wideband integrals on the sphere}

\date{\today}

\numberwithin{equation}{section}
\begin{document}
\maketitle

\begin{abstract} We propose, analyze, and implement  interpolatory
approximations and Filon-type cubature for  efficient and accurate evaluation of
a class of wideband generalized Fourier integrals on the sphere. The  analysis
includes derivation of (i)  optimal order Sobolev norm error estimates for an
explicit discrete Fourier transform type interpolatory approximation of
spherical functions; and (ii) a wavenumber explicit error estimate of the order
$\mathcal{O}(\kappa^{-\ell} N^{-r_\ell})$, for $\ell = 0, 1, 2$, where $\kappa$
is the wavenumber, $N$ is the number of interpolation/cubature points on the
sphere and $r_\ell$ depends on the smoothness of the integrand. Consequently,
the  cubature  is robust for wideband  (from very low to very high) frequencies
and very efficient for highly-oscillatory integrals  because the quality of the
high-order approximation (with respect to quadrature points) is further improved
as the wavenumber increases. This  property is a marked advantage compared to
standard cubature that require at least ten points per wavelength per dimension 
and   methods for which asymptotic convergence is known only with respect to the
wavenumber subject to stable of computation of quadrature weights. Numerical
results in this article  demonstrate the  optimal order accuracy of the
interpolatory approximations and the wideband cubature. 

\end{abstract}

{\bf {Key words:}} 
Interpolation, cubature, wideband Fourier integrals, sphere

\vspace{0.2in}       
{\bf {AMS subject classifications:}}
42A15, 65D32, 33C55

\vspace{0.2in}                                                                  
                             
{\bf {Running title:}} Interpolation and cubature  approximations on the sphere

\section{Introduction}\label{sec:intro}

The generalized  Fourier integral operator $I_{\kappa}^g$, for a fixed phase
function $g$ and  wavenumber $\kappa \in (0, \infty)$, has the
representation~\cite{BH86, Wong89}  
\begin{equation}\label{eq:gen_int}
I_{\kappa}^g f  = \int_\Omega f(\x) \exp({\rm i} \kappa g(\x))  \; {\rm d}\x.
\end{equation} 
Such integrals occur in various applications~\cite{BH86,
wideband:1, GH:11, GaLaSl:07,  wideband:2, wideband:3, Wong89} including
simulation of waves scattered by the surface of an obstacle when an incident
wave (say, a plane wave $\exp({\rm i} \kappa \x \cdot \ddh)$)  strikes the
obstacle from direction $\ddh$. In wave propagation models,  for $\x \in
\Omega$, the phase  $g(\x)$ depends on the incident direction $\ddh$ and the
observed point $\y \in \Omega$. Several fast algorithms and
applications~\cite{BH86, wideband:1, GH:11,   wideband:2, wideband:3}  require
accurate and efficient  evaluation of such integrals for large number of
observed points and also for  wideband frequencies (that is, for $\kappa$ in
\eqref{eq:gen_int}  taking  a wide range of values  from very small to
very large).
 
Standard discretization techniques to evaluate such large numbers of
generalized
Fourier integrals is prohibitive, in particular  for high-frequency wavenumber
$\kappa$  
and for $\Omega \subset \mathbb{R}^3$.  This is mainly because standard
quadrature 
rules  require at least ten points per wavelength per dimension, even for 
low-order accuracy. 
This article is motivated by the recent work of 
Dom\'{\i}nguez et al.~\cite{DoGrSm:2010} on
evaluating~\eqref{eq:gen_int} in one dimension 
(with $\Omega = [-1, 1]$  and $g(s) = s$ on $\Omega$) 
and of   Ganesh and Hawkins~\cite{GH:11} 
on a fast algorithm to simulate high-frequency exterior acoustic scattering in
three dimensions.

In addition to the classical asymptotic methods~\cite{BH86, Wong89}, there is a
large literature on efficient numerical evaluation of~\eqref{eq:gen_int},
see~\cite{IsNo:05, Ol:1} and  references therein. Such quadrature rules can be  
classified as Filon-type, Levin-type, or numerical steepest descent   methods. A
 recent effort to avoid stability issues in such methods, for a one dimensional
domain $\Omega$, is the   shifted GMRES  method~\cite{Ol:2} that requires
solution of first order differential equations to evaluate $I_{\kappa}^g f$. The
main aim in this literature is to prove asymptotic convergence with
respect to
the wavenumber $\kappa$ and  address stability issues associated with
computing weights for large $\kappa$. We refer to the recent work~\cite{Ol:2,
Ol:1} for details of the need to develop efficient methods to
compute~\eqref{eq:gen_int} even with $\Omega = [-1, 1]$ and discussions
regarding the type of approach one needs to adopt depending on the properties of
the phase function $g$.  In particular, Filon-type methods are  efficient if the
phase function $g$ does not have stationary  points.

For wave propagation  applications, whether the phase function $g$ has
stationary points or not depends on the  observation points. For solving the
full obstacle scattering wave propagation models, observation points should be
considered in all regions of the scatterer. Consequently, various types of
techniques are required to evaluate~\eqref{eq:gen_int} depending, for example,
on whether the observation point is near or away from the shadow boundary and
whether the corresponding phase function has stationary~\cite{BH86, Wong89} or
steepness~\cite{GH:11}  points and also whether the density function $f$ has
singular points in $\Omega$. 

For a boundary integral equation reformulation of the wave propagation model
exterior to a convex scatterer, efficient approaches to evaluate $I_{\kappa}^g$
for stationary, steepness, and  singular point cases are discussed
in~\cite{Kim:thesis, GH:11} (and references therein), respectively,  for two 
and three dimensional cases.  In such applications, $\Omega$
in~\eqref{eq:gen_int} is the boundary of the convex obstacle, which can be
diffeomorphically mapped onto $\mathbb{S}^{n-1}$, the unit sphere in
$\mathbb{R}^n$, $n=2,3$.  Such a global mapping property was used
in~\cite{Kim:thesis, GH:11} to transplant the wave propagation model  to a
boundary integral equation on  $\mathbb{S}^{n-1}$, $n = 2, 3$ and linear
combinations of appropriate  polynomial basis functions are used to approximate
the unknown surface current.

The focus of the  recent work of Dom\'{\i}nguez et al.~\cite{DoGrSm:2010} is to 
evaluate $I_{k}^g$ for a two dimensional scattering model~\cite{Kim:thesis} for
the case when $g$ does not have stationary points and $f$ does not have singular
points in $\Omega$.  This leads to the requirement of
evaluating~\eqref{eq:gen_int} with $\Omega = [-1, 1]$  and $g(s) = s,~s \in
[-1,1]$. We note that the integral considered in~\cite{DoGrSm:2010} is
equivalent to~\eqref{eq:gen_int} with $\Omega = \mathbb{S}^{1}$ and $g(\x) = \x
\cdot \ddh,~ \x \in  \mathbb{S}^{1}$ when the incident wave direction is
assumed, without loss of generality (due to rotational symmetry of
$\mathbb{S}^{1}$),  to be $\ddh = [0, 1]^T$.  The main aim of this article is to
carry out the three dimensional counterpart of the method and analysis
in~\cite{DoGrSm:2010}, by taking $\Omega = \mathbb{S}^{2}$ and $g(\x) = \x \cdot
\ddh,~ \x \in  \mathbb{S}^{2}$  in~\eqref{eq:gen_int} with $\ddh = [0, 0,
1]^T$. 

For the sphere case, even  basic tools that are used to
accomplish an efficient
method and analysis in~\cite{DoGrSm:2010} is missing in the literature. A major
contribution in this article is to first develop such  
tools and analysis, and hence derive an efficient
Filon-type cubature and associated error estimates to
evaluate the class of wideband integrals on the sphere defined by
\begin{equation}\label{eq:gen_int_sph}
I_{\kappa} f  = \int_{\sph} f(\x) \exp({\rm i} \kappa \x \cdot \ddh)  \;
{\rm d}\x, \qquad \qquad \ddh =  [0, 0, 1]^T,
 \qquad \qquad \kappa \in (0, \infty).
\end{equation}

The Filon-type quadrature rules (and wavenumber explicit analysis) to evaluate
the generalized Fourier integral is based on the idea of  replacing the density
function $f$ by an interpolatory approximation $Q_Nf \in X_N$
that can be
efficiently computed. In addition for wavenumber  explicit analysis, associated
interpolatory approximation errors in   Sobolev norms are
required~\cite{DoGrSm:2010}. The choice of the finite dimensional space $X_N$,
associated basis functions,  interpolation points and simple representation of 
$Q_Nf$ are crucial for efficient and stable approximations of the generalized
Fourier integrals and analysis.

For example, in~\cite{DoGrSm:2010}, with $X_N$ being the space of all
polynomials
of degree at most $N$ on $[-1, 1]$, 
\begin{equation}\label{eq:gen_int_sph_oned}
\widetilde{I_{\kappa}} f  = \int_{-1}^1 f(x) \exp({\rm
i} \kappa x)\, {\rm d}x   \approx
\widetilde{I_{\kappa,N}} f  := \int_{-1}^1 (Q_N f)(x)
\exp({\rm i} \kappa x)  \; {\rm d}x,
\end{equation} 
where $Q_N f$ interpolates $f$ at $N+1$ (quadrature) points  in
$[-1, 1]$ that are Chebyshev points $t_{j,N} = \cos(j \pi /N),~j=0, \ldots, N$.
In addition, $Q_N f$ can be explicitly expressed using the discrete Fast Fourier
Transform (DFFT)  type linear combination of the Chebyshev polynomial
basis
functions $T_n,~n= 0, \ldots, N$:
\begin{equation}\label{eq:oned_fft_rep}
Q_N f= \sum_{n=0}^N{}^{\prime \prime} \langle f , T_n \rangle_N T_n,  \qquad 
 \langle f , T_n \rangle_N =  \frac{2}{N} \sum_{j=0}^N{}^{\prime \prime}  \cos(j
n \pi /N) f(t_{j,N}),
\end{equation}
where $\sum{}^{\prime \prime}$ means the first and last terms in the sum are to
be halved.
In~\eqref{eq:oned_fft_rep}, $\langle \cdot, \cdot \rangle_N$ is the discrete
inner product
quadrature approximation to the weighted  $L^2$ inner product
\[
\langle f, g \rangle:=\int_{-1}^1 f(x)\overline{g(x)}\,\frac{{\rm
d}x}{\sqrt{1-x^2}}=\int_0^\pi f(\cos\theta)\overline{g(\cos\theta)}\,{\rm
d}\theta
\]
with a crucial property that
\begin{equation}\label{eq:dip_prop}
\langle \psi_1, \psi_2 \rangle_N  = \langle \psi_1, \psi_2 \rangle, \qquad
\qquad \psi_1, \psi_2 \in X_N.
\end{equation}
The interpolation operator with explicit DFFT-type representation and quadrature
property
in~\eqref{eq:oned_fft_rep}-\eqref{eq:dip_prop}
does not require solving any matrix equation. Such an operator  is known as 
the DFFT-type matrix-free interpolation operator~\cite{GaMha:2006}. 
The error analysis in~\cite{DoGrSm:2010} crucially depends
on the classical error estimates of approximating $2\pi$-periodic functions in 
the Sobolev space 
$H^\mu(-\pi, \pi)$, for $\mu = 0,1$: With $\psi_c(\theta) = \psi(\cos \theta)$,
for 
$0\leq \mu \leq \nu$,  $\nu > 1/2$, there is a constant $C_{\nu, \mu}$ such
that~\cite{DoGrSm:2010} 
\begin{equation}\label{eq:oned_est}
\| f_c - (Q_N f)_c \|_{H^\mu
(-\pi, \pi)} \leq C_{\nu, \mu} N^{\mu-\nu} \| f_c \|_{H^\nu(-\pi, \pi)}, \qquad
\mu = 0, 1.
\end{equation}
For the sphere case, with $X_N$ being the $(N+1)^2$ dimensional space of all
spherical polynomials of degree at most $N\geq 3$, it is impossible to construct
a DFFT-type matrix-free interpolation operator~\cite{sloan:1}. This seminal work
of Sloan resulted in addressing several fundamental questions related to
polynomial interpolation and numerical integration on the sphere,
see~\cite{atk_book,  sloan:3, sloan:2} and extensive references therein.

Using a new class of finite dimensional spaces $X_N$ of spherical functions
(some of which are not polynomials), Ganesh and Mhaskar~\cite{GaMha:2006} 
constructed matrix-free interpolation operators on the sphere. However,
the analysis
in~\cite{GaMha:2006} does not contain the Sobolev norm error estimate of the
DFFT-type interpolation operators. In particular, for interpolation operator
considered in this article we are not aware of error analysis in any norm.
Deriving the Sobolev error estimates similar to~\eqref{eq:oned_est} 
is one of the main contributions of this article.

This paper is organized as follows.  
In the next section   using a finite dimensional space
introduced in~\cite{GaGrSi:1998, GaMha:2006} we prove that 
the interpolation problem onto the space is well defined for any set of $N+1$
points in the latitudinal (elevation) angles in $[0, \pi]$ (that include the
poles) provided that $2N$ equally spaced points are chosen in  the longitudinal
(azimuth) angle interval $(-\pi, \pi]$. In Section~\ref{sec:sob_space} we study
Sobolev-like Fourier spaces that orthogonally decompose the standard Sobolev
spaces on the sphere and establish properties of functions in such spaces.
(Proofs of such technical results are deferred to
Appendix~\ref{sec:fourier_space_appendix}.) In Section~\ref{sec:optim_conv} we
prove optimal order convergence, similar to~\eqref{eq:oned_est}, 
of a DFFT-type matrix-free interpolation operator induced by the
Gauss-Lobatto latitudinal quadrature points. (Major part of  proofs in
Section~\ref{sec:optim_conv} are deferred to
Appendix~\ref{sec:appB_fourier_anal_proof}.) In Section~\ref{sec:cub_anal} we
introduce a Filon-type cubature to evaluate the wideband
integrals~\eqref{eq:gen_int_sph} and prove that the associated error  is of the
order $\mathcal{O}(\kappa^{-\ell} N^{-r_\ell})$, for $\ell = 0, 1, 2$ where
$r_\ell$ depends on the smoothness of the density function $f$. Numerical
results in Section~\ref{sec:num_exp} demonstrate the robust theoretical analysis
and efficiency of evaluating the wideband integrals for several low to high
frequencies. 

\section{Interpolation on the sphere}\label{sec:interp_def}

In this section we introduce a class of interpolatory approximations to 
continuous functions on $\Sp$. Following the standard convention, we identify
functions on the sphere  in terms of functions defined using the spherical polar
coordinates, 
\begin{equation}\label{eq:F_Fcirc}
 F(\theta,\phi):=\big(F^\circ\circ \p\big) (\theta,\phi),
\end{equation} 
where we used the natural  spherical polar coordinates parameterization
\begin{equation}
\label{eq:pmap}
\xh = \p(\theta,\phi)= (\sin \theta \cos \phi,\sin \theta \sin \phi,\cos
\theta)^T, \qquad \xh \in \sph.
\end{equation}
Throughout this article, any function defined on the unit sphere $\Sp$ will
be superscripted with ``$\circ$''.

Using~\eqref{eq:pmap}, we obtain the following periodic and reflectional
symmetric properties of $F$ in~\eqref{eq:F_Fcirc}
\begin{equation}
\label{eq:Fstar1}
F(\theta,\phi)=F(\theta+2\pi,\phi)=F(\theta,\phi+2\pi), \qquad \qquad
F(-\theta,\phi+\pi)=F(\theta,\phi).
\end{equation}
Let ${\cal C}(\Sp)$ be the space of all continuous functions on $\Sp$ and define
\begin{equation}
 {\cal C}:=\big\{F:\R^2\to \mathbb{C}\ \big|\ \mbox{ $F$
satisfies~\eqref{eq:F_Fcirc}
for some}~
 F^\circ\in {\cal C}(\Sp)\big\}.\label{eq:contsp}
\end{equation}
Observe that if $F\in{\cal C}$ then
\begin{equation}
\label{eq:Fstar2}
 F(0,\, \cdot\,),\quad F(\pi,\, \cdot\,)\text{\quad are constants}.
\end{equation}
Conversely,  if $F:\R^2\to\C$ is continuous and
satisfies \eqref{eq:Fstar1} and \eqref{eq:Fstar2} then there exists a
unique $F^\circ\in{\cal C}(\Sp)$
such that $F=F^\circ \circ \p$. Hence, ${\cal C} \cong
{\cal C}(\Sp)$.

We first consider a finite dimensional subspace
$\chi_N$ of ${\cal C}$,
depending on a parameter $N$, in which we shall seek interpolatory
approximations.
Such a  space has been introduced and studied in~\cite{GaGrSi:1998, GaMha:2006}.
We refer to details in~\cite{GaGrSi:1998, GaMha:2006} for arriving at the
space  by restricting the space of all bivariate trigonometric polynomials
of degree at $N$ to those functions that satisfy the essential spherical
function
properties~\eqref{eq:Fstar1} and \eqref{eq:Fstar2}.

For  $M\in\mathbb{N}$, we consider the finite dimensional  spaces
\begin{equation}
\begin{array}{rcl}
{\mathbb D}_M^{\rm e}&:=&{\rm
span}\big\{ \cos n\theta\ :\ n=0,\ldots, M\big\} ={\rm
span}\big\{ \cos^n\theta\ :\ n=0,\ldots, M\big\},
\\
{\mathbb D}_{M}^{{\rm o}}&:=& {\rm
span}\big\{ \sin\theta \cos n\theta\ :\ n=0,\ldots, M\big\} \\
&=&{\rm
span}\big\{ \sin\theta \cos^n \theta\ :\ n=0,\ldots, M\big \}.
\end{array}
\label{eq:def:DNeDNo}
\end{equation}
Throughout this article $N \in \mathbb{N}$ with $N \geq 2$.  
For $p\in\mathbb{D}_N^{\rm e}$  there exist $q\in\mathbb{D}_{N-2}^{\rm e}$
and $r\in\mathbb{D}_{1}^{\rm e}$ such that $p(\theta) = r(\theta)+ \sin^2\theta\
q(\theta)$. Hence
if $p\in\mathbb{D}_N^{\rm e}$, we obtain
\[
p(0)=p(\pi)=0\quad \Longleftrightarrow\quad p (\theta)=\sin^2\theta\
q(\theta),\qquad\text{with } q\in\mathbb{D}^{\rm e}_{N-2}.
\] 

Following~\cite{GaGrSi:1998, GaMha:2006}, we consider 
the finite dimensional subspace  $\chi_N\subset {\cal C}$:
\begin{eqnarray}
 \chi_N&:=&
\bigg\{p_0(\theta)+\!\!\sum_{ -N  <m \le N\atop  \text{even }m\ne0}\!\!\!\!
\sin^2\theta\, p_m(\theta)\exp({\rm i}m\phi) +\!\!
\sum_{ -N <m \le N\atop   \text{odd } m}\!\!
  p_m(\theta)\exp({\rm i}m\phi)\ \bigg|\nonumber\\
&&
 \ \hspace{3cm}\ p_0\in {\mathbb D}^{\rm e}_N,\
  p_{2\ell}\in
{\mathbb D}^{\rm e}_{N-2},\ \ell \ne 0, \ p_{2\ell+1}\in
{\mathbb D}^{\rm o}_{N-2}\bigg\}\label{eq:chin}\\
&=& \bigg\{\sum_{ -N  <m \le N\atop  \text{even }m}\!\!\!\!
  p_m(\theta)\exp({\rm i}m\phi) +\!\!
\sum_{ -N <m \le N\atop   \text{odd } m}\!\!
  p_m(\theta)\exp({\rm i}m\phi)\ \bigg|\nonumber\\
&&
 \ \hspace{2cm}\ p_{2\ell}\in {\mathbb D}^{\rm e}_N,\
  p_{2\ell}(0)=p_{2\ell}(\pi)=0 \text{ for $\ell \ne 0$}, \ p_{2\ell+1}\in
{\mathbb D}^{\rm o}_{N-2}\bigg\}.\ \quad\label{eq:chin2}
\end{eqnarray}
It is easy to see that the dimension of $\chi_N$ is $2N^2-2N+2$.

Let  $0= \xi_0<\xi_1<\cdots<\xi_N=\pi$ with arbitrarily chosen set of 
$N-1$ latitudinal angles in $(0, \pi)$. In this article, we choose $2N$ 
equally spaced azimuthal angles  $\phi_k = k\pi/N,~k = -N+1, \ldots, N$.
These sets induce a discrete set of $2N^2 - 2N +2$ points on the unit sphere, by
noticing from \eqref{eq:pmap} that $\p(\xi_0, \phi_k)$ and $\p(\xi_N, \phi_k)$
are respectively
the north and south pole, for $k = -N+1, \ldots, N$.

Using these discrete sets of coordinates,  we
introduce the interpolation problem: For any $F\in{\cal C}$,
find ${\cal Q}_NF \in \chi_N$
such that
\begin{equation}\label{eq:nodesInterpolationAbstr}
  \big({\cal Q}_N
F\big)(\xi_j,{\textstyle
\frac{k\pi}{N}})=F({\textstyle\xi_j,\frac{
k\pi}{N}}),\quad
j=0,\ldots,N,\ \ k=-N+1,\ldots,N.
\end{equation}
For any $F\in{\cal C}$, since $F(0,\, \cdot\,),~F(\pi,\, \cdot\,)$ are
constants, 
\eqref{eq:nodesInterpolationAbstr} is equivalent to only $2N^2-2N+2$
interpolation
conditions. Next we constructively show that the interpolation with the
arbitrary $N-1$ elevation
angles in $(0, \pi)$ is well-posed.  Such a construction will also play a key
role in
developing an efficient approximation of the wideband
integrals~\eqref{eq:gen_int_sph}.

\begin{proposition}\label{prop:InterpWellDefined}
For any $F\in{\cal
C}$, the interpolation problem
\eqref{eq:nodesInterpolationAbstr}  has a unique solution.
\end{proposition}
\begin{proof} It is sufficient  to show that such an
interpolant can be always constructed.
Set
\[
 F^j_k:=F({\textstyle \xi_j,\frac{ k\pi}{N}}),\quad j=0,\ldots,N,\ \
k=-N+1,\ldots, N,
\]
and  define  for $j=0,\ldots, N$
\begin{eqnarray}
\label{eq:FT}
 f^j_m&:=&\frac{1}{2N}\sum_{-N<k\le N} F^j_k\exp({\textstyle -\frac{2\pi m  k
{\rm i}}{2N}}), \qquad m=-N+1,\ldots,N.
\end{eqnarray}
The vector ${\bf f}^j:=(f_m^j)_{m=-N+1,\ldots,
N}\subset\mathbb{C}^{2N}$ is indeed (a symmetric variant of)  the
result of applying the inverse finite Fourier transform to the vector
${\bf F}^j:=(F_k^j)_{k=-N+1,\ldots,N}\subset\mathbb{C}^{2N}$. Hence,
${\bf F}^j$ can be recovered from ${\bf f}^j$ via
\begin{equation}\label{eq:FFT}
F^j_k=\sum_{-N<m\le N} f^j_m\exp({\textstyle \frac{2\pi m k
 {\rm i} }{2N}}), \qquad k=-N+1,\ldots,N.
\end{equation}
Using \eqref{eq:Fstar1}, we derive
\[
  F_k^0=F\big(0, {\textstyle\frac{k\pi}{N}}\big)=F(0,0),\quad
  F_k^N=F\big(\pi,  {\textstyle\frac{k\pi}{N}} \big)=F(\pi,0),\qquad
k=-N+1,\ldots,N, 
\]
and consequently
\begin{equation}
\label{eq:fm0N}
 f_m^0=f_m^N=0,\quad \text{for all }m\ne 0,\qquad f_0^0=F(0,\cdot),\qquad
f_0^N=F(\pi,\cdot).
\end{equation}
The interpolation problem: Find  $p_{2\ell}\in
\mathbb{D}_N^{\rm e}$ 
such that
\[
p_{2\ell}(\xi_j)=f_{2\ell}^j,\qquad j=0,\ldots, N, 
\]
is equivalent, with the change of variables $x=\cos\theta$, to   a uniquely
solvable polynomial interpolation problem on $[-1,1]$. Moreover, \eqref{eq:fm0N}
implies that for $\ell\ne
0$,
\begin{equation}
p_{2\ell}(0)=p_{2\ell}(\pi)=0.
\end{equation}
Similarly, for the odd coefficients,  we solve the well-posed interpolation
problem: Find  $\widetilde{p}_{2\ell + 1}\in
\mathbb{D}_{N-2}^{\rm e}$ 
such that
\[
\widetilde{p}_{2\ell+1}(\xi_j)=\frac{f_{2\ell+1}^j}{\sin\xi_j},\qquad
j=1,\ldots,
N-1,
\]
and set  
$p_{2\ell+1}:=\sin(\,\cdot\,)\,\widetilde{p}_{2\ell+1}\in\mathbb{D}_{N-2}
^{\rm o}$. Hence we obtain 
\[
  p_{2\ell+1}\in \mathbb{D}_{N-2}^{\rm o}, \qquad
p_{2\ell+1}(\xi_j)=f_{2\ell+1}^j,\qquad j=0,\ldots, N.
\]
If we define 
\[
F_N(\theta,\phi):=\sum_{-N <m\le N} p_m(\theta)\exp({\rm i}m\phi)\in\chi_N,
\]
using~\eqref{eq:FFT}, we obtain
\begin{eqnarray*}
F_N({\textstyle \xi_j,\frac{ k\pi}{N}})&=&\sum_{-N <m\le N} p_m(\xi_j)\exp(
{\textstyle \frac{2\pi m k
 {\rm i} }{2N}}) =\sum_{-N <m\le N} f^j_m \exp(
{\textstyle \frac{2\pi m k
 {\rm i} }{2N}}) =F_k^j=F({\textstyle \xi_j,\frac{ k\pi}{N}}).
\end{eqnarray*}
That is, ${\cal Q}_NF = F_N$ is the unique solution of the 
interpolation problem~\eqref{eq:nodesInterpolationAbstr}. 
\end{proof}

Note that for the well-posed interpolation
problem~\eqref{eq:nodesInterpolationAbstr},   {\em a priori}, 
we are free to choose the $N-1$ nodes at the elevation angle.
The quality of the interpolant depends crucially on the choice of
the latitudinal interpolation points  $0= \xi_0<\xi_1<\cdots<\xi_N=\pi$.
In particular, for developing
efficient Filon-type cubature to evaluate the wideband 
integrals~\eqref{eq:gen_int_sph}, we require an interpolation operator  ${\cal
Q}_N$ on the sphere 
with optimal order convergence properties, similar to~\eqref{eq:oned_est}, 
in the Sobolev ${\cal H}^s$ norm for $s = 0, 1$. (Definitions  of  these
norms are
presented in next section.)

That is, we require the Lebesgue constant of the interpolation operator ${\cal
Q}_N$  to be uniformly bounded in ${\cal H}^0$ norm and grow only by
$\mathcal{O}(N)$ in the   ${\cal H}^1$ norm. 
In addition for efficient evaluation, similar to \eqref{eq:oned_fft_rep},  
it is ideal to construct  ${\cal Q}_N F$ with a DFFT-type explicit
representation.

The simple  choice of $\xi_j,~j=0, \ldots, N$ being equally spaced points 
in $[0, \pi]$ was considered in~\cite{GaGrSi:1998, GaMha:2006} and the Lebesgue
constant
of the associated interpolation  operator  ${\cal Q}_N^{\rm u}$ on the sphere 
in the uniform norm (that is, the standard norm in ${\cal C}$) 
was proved in detail to be  $\mathcal{O}(\log^2 N)$. The work
in~\cite{GaMha:2006} in addition 
includes developing a grid-specific
discrete inner product in $\chi_N$ and an associated orthonormal basis and hence
a  
DFFT-type representation of ${\cal Q}_N^{\rm u}$,
see~\cite[Theorem~2.3]{GaMha:2006}.
While the space $\chi_N$ in~\eqref{eq:chin} is same as that
in~\cite{GaGrSi:1998}
 (with  index $m$ satisfying $-N <m \leq N$) 
the finite dimensional subspace of ${\cal C}$ in~\cite{GaMha:2006} is slightly
different 
(with index $m$ satisfying $|m| \leq N$). Results in~\cite{GaMha:2006} hold with
 $|m| \leq N$ replaced with  $-N <m \leq N$.

A DFFT-type matrix-free interpolation operator ${\cal Q}_N^{\rm
gl}$ on the sphere with the $N+1$ latitudinal points  based on the Gauss-Lobatto
points
was also introduced in~\cite[Theorem 2.2]{GaMha:2006}: Let
$\{\eta_{j}\}_{j=0}^N$ be   the nodes of the Gauss-Lobatto quadrature
rule in $[-1,1]$. That is, $\eta_0=-1$, $\eta_N=1$ and $\eta_j$ for
$j=1,\ldots,N-1$ are the roots of $P_{N}'$, where $P_N$ is the
Legendre polynomial of degree $N$~\cite{atk_book}. 
The $N+1$ latitudinal Gauss-Lobatto  points are
\begin{equation}
\label{eq:deftheta}
 \xi_j= \theta_j := \arccos(\eta_{N-j}),\qquad j=0,\ldots, N.
\end{equation}

Thus ${\cal Q}^{\rm
gl}_N:{\cal C}\to \chi_N$ 
is defined, for $F \in {\cal C}$,   as the solution of 
\begin{equation}
\label{eq:nodesInterpolation}
\big({\cal Q}^{\rm gl}_N F\big)
(\theta_j,{\textstyle
\frac{k\pi}{N}})=F({\textstyle\theta_j,\frac{
k\pi}{N}}),\qquad
j=0,\ldots,N,\quad k=-N+1,\ldots,N.
\end{equation}
For a DFFT-type matrix-free analytical representation of ${\cal Q}^{\rm gl}_N$,
we refer to~\cite[Equation~(2.23)]{GaMha:2006}.
In this article  we use  ${\cal Q}_N^{\rm gl}$ for developing
an efficient cubature rule for the wideband integrals.
 
Unlike the ${\cal Q}^{\rm u}_N$ case, convergence
properties of ${\cal Q}^{\rm gl}_N$ have not been established in any
norm~\cite{GaMha:2006}. 
One of the key contributions in this article is to solve this open problem.
In Section~\ref{sec:optim_conv} we will prove optimal order convergence
properties of  ${\cal
Q}_N^{\rm gl}$ in the Sobolev  ${\cal H}^s$ norm for $s = 0, 1$, after
establishing various 
results in the next section. The convergence properties 
will be used  in Section~\ref{sec:cub_anal}  to prove robust
error estimates for a new class of cubature rules on
the sphere, with respect to the  wavenumber and quadrature points.

\section{Sobolev and Sobolev-like  Fourier coefficient
spaces}\label{sec:sob_space}

The standard Sobolev spaces  on the sphere ${\cal H}^s(\Sp)$, can be introduced
in several, but equivalent,  ways. For instance, we can start from a smooth
atlas for $\Sp$, i.e. local charts with a subordinated partition of unity,  and
define the space in terms of the appropriate  Sobolev space of functionals on
$\R^2$. The space so defined is independent of the chosen atlas used to describe
the surface~\cite{ad:2003,Ne:2001}.

An alternative approach is based on using  the eigensystem of the
Laplace-Beltrami operator on the sphere and construct a Hilbert scale which
gives rise to the Sobolev spaces for any range. We follow the latter approach. 
For a sufficiently smooth function $F^\circ:\Sp\to\mathbb{C}$, we consider the
parameterized surface gradient and Laplace-Beltrami operators  
\begin{eqnarray*}
 \nabla_{\Sp} F(\theta,\phi)&=&\frac{1}{\sin\theta}\frac{\partial
F}{\partial \phi}(\theta,\phi) \frac{\partial
\p}{\partial \phi}(\theta,\phi)+\frac{\partial
F}{\partial \theta} (\theta,\phi) \frac{\partial
\p}{\partial \theta}(\theta,\phi),\\
 \Delta_{\Sp} F(\theta,\phi)&=&\frac{1}{\sin^2\theta}\frac{\partial^2
F}{\partial \phi^2}(\theta,\phi)
+\frac{1}{\sin\theta}\frac{\partial}{\partial\theta}\Big[\sin\theta\frac{
\partial
F}{\partial\theta}(\theta,\phi)\Big].
\end{eqnarray*}

Corresponding to  ${\cal L}^2(\Sp)$, the 
space of all square integrable functions on the sphere, we define 
the equivalent space
\[
{\cal H}^0
=\Big\{F:\R^2\to\mathbb{C}\
\Big|\
F\mbox{ satisfies  \eqref{eq:Fstar1}--\eqref{eq:Fstar2}} , \quad
\|F\|_{{\cal H}^0}<\infty
\Big\},
\]
with 
\begin{equation}\label{eq:normH0}
\|F\|_{{\cal H}^0}^2:= \int_0^\pi\!\!
\int_0^{2\pi}|F(\theta,\phi)|^2\sin\theta\,{\rm d}\theta\,{\rm d}\phi. 
\end{equation}
We have just  used the fact that $\sin\theta {\rm
d}\theta\,{\rm d}\phi$ is the surface
area element on the unit sphere.

Similarly we define
\[
 {\cal H}^1:=\big\{F:\R^2\to\mathbb{C}\ \big|\ F\in {\cal H}^{0}, \nabla_{\Sp}
F\in
{\cal H}^{0}\times {\cal H}^{0} \big\}
\]
equipped with the norm
\begin{eqnarray}
 \|F\|_{{\cal H}^1}^2&:=&\frac{1}{4}\|F\|_{{\cal
H}^0}^2+\|\nabla_{\Sp}F\|_{{\cal  H}^0\times {\cal  H}^0}^2\nonumber\\
&=&\frac{1}{4}\int_{0}^\pi\!\int_{0}^{2\pi} |F(\theta,
\phi)|^2\sin\theta\,{\rm d}\theta\nonumber\\
&&+\int_{0}^\pi\!\int_{0}^{2\pi}
\bigg|\frac{\partial
F}{\partial \phi}(\theta,\phi)\bigg|^2\frac{1}{\sin\theta}{\rm
d\theta}\,{\rm d}\phi+\int_{0}^\pi\!\int_{0}^{2\pi}
\bigg|\frac{\partial
F}{\partial \theta}(\theta,\phi)\bigg|^2\ {\sin\theta}\,{\rm d\theta}\,{\rm
d}\phi.  \label{eq:normH1}
\end{eqnarray}

In order to introduce the equivalent high order Sobolev spaces through
the eigensystem of the Laplace Beltrami operator,  we consider
the associated Legendre functions~\cite{atk_book,Ne:2001}
\begin{equation}
\label{eq:Legendres}
P_n^m(x)=
\frac{(-1)^n}{2^n n!} (1-x^2)^{m/2}\frac{{\rm d}^{m+n}}{{\rm
d}x^{m+n}}(1-x^2)^n.
\end{equation}
  Define
\begin{equation}\label{eq:qLeg}
Q_n^m(\theta):=\bigg(\frac{2n+1}2 \frac{(n-m)!}{(n+m)!}
\bigg)^{1/2}P_n^{|m|}(\cos \theta),\quad
Q_{n}^{-m}:=Q_{n}^m, \qquad 0\le m\le n,\quad n=0,1,\ldots
\end{equation}
and  set
\begin{equation}\label{eq:em}
e_m(\phi) :=\frac{1}{\sqrt{2\pi}}\exp({\rm i}m\phi),\qquad m\in\Z.
\end{equation}
The spherical harmonics~\cite{atk_book,Ne:2001},  polynomials  of degree $n$ 
on the sphere, can be defined as
\begin{equation}
\label{eq:defHarmonicSpherics}
Y_n^m(\theta,\phi):= (-1)^{(m+|m|)/2} Q_{n}^m( \theta) e_{m}(\phi),\qquad
m=- n ,\ldots,n,\quad n=0,1,\ldots
\end{equation}
It is well known that $\{Y_n^m \ : \  n = 0, 1, 2, \ldots, |m| \leq n\}$ 
is an orthonormal basis of ${\cal H}^0\cong
{\cal L}^2(\Sp)$ and~\cite[(2.4.175)]{Ne:2001}
\begin{equation}
 \label{eq:orthogonality}
 \int_{0}^{2\pi}\int_{0}^{\pi} \nabla_{\Sp} Y_{n}^m(\theta,\phi)
\cdot \overline{\nabla_{\Sp} Y_{n'}^{m'}(\theta,\phi)} \sin \theta \,{\rm
d}\theta\,{\rm d\phi}=\left\{\begin{array}{ll}
                 n(n+1),&\mbox{if $(n,m)=(n',m')$},\\
0,&\mbox{otherwise}.
                \end{array}\right.
\end{equation}
The above result follows from the fact that  $ \{Y_n^m\}$ are  eigenvectors of 
$\Delta_\Sp$~\cite[(2.4.23)]{Ne:2001}:
\begin{equation}
 \label{eq:EigenLaplacian}
-\Delta_\Sp Y_{n}^m=\big(n(n+1)\big) Y_{n}^m,\qquad
m\in\mathbb{Z},\
n=|m|,\ldots, \infty. 
\end{equation}
Thus  $\{(n+1/2)^{-1}Y_n^m \:|\:  n = 0, 1, 2, \ldots, |m| \leq n \}$ 
is  an orthonormal basis for ${\cal H}^1$. 

For any $s\in\mathbb{R}$ we introduce the norm
\begin{equation} \label{eq:Sobolevnorm}
 \|F\|_{{\cal H}^s}^2:=  \sum_{n=0}^\infty
\sum_{m=-n}^n
\big(n+{\textstyle\frac{1}{2}}\big)^{2s}  |\widehat{F}_{n,m}|^2, \quad
\end{equation}
where
\begin{equation}
 \label{eq:defFnm}
 \widehat{F}_{n,m}:=\int_0^\pi\int_0^{2\pi}F (\theta,\phi)
\overline{{ Y_{n}^m(\theta,\phi)}}\sin\theta\,{\rm d}\theta\,{\rm d}\phi.
\end{equation}
The Sobolev spaces ${\cal H}^s$ for $s\ge 0$ are then defined as
\[
 {\cal H}^s:=  \big\{F\in {\cal H}^0\ \big|\ \|F\|_{{\cal H}^s}<\infty\big\}.
\]

Observe that for any $F\in {\cal H}^s$,
\begin{equation}
 \label{eq:series}
 F=
\sum_{n=0}^\infty\sum_{m=-n}^n  \widehat{F}_{n,m} Y_n^{m}, 
\end{equation}
with the series converging  now in ${\cal H}^s$. Thus,   ${\rm
span} \big\{Y_n^m : n = 0, 1, 2, \ldots, ~~|m| \leq n\big\}$ is
dense in ${\cal H}^s$ for any $s$.

with
We point out that for $s=0,1$, the norms in
\eqref{eq:Sobolevnorm}  and in
\eqref{eq:normH0}-\eqref{eq:normH1} 
coincide. Moreover, 
\begin{eqnarray}
 \|F\|_{{\cal H}^2}^2&=&\|\Delta_\Sp F\|_{{\cal H}^0}^2+\frac{1}4\|\nabla_\Sp
F\|_{{\cal H}^0\times {\cal
H}^0}^2+\frac{1}{16}\|F\|_{{\cal H}^0}^2\nonumber\\
&=&
\int_{0}^\pi\!\int_{0}^{2\pi} \bigg|\frac{1}{\sin^2\theta}\frac{\partial^2
F}{\partial \phi^2}(\theta,\phi)
+\frac{1}{\sin\theta}\frac{\partial}{\partial\theta}\Big(\sin\theta\frac{
\partial
F}{\partial\theta}(\theta,\phi)\Big)\bigg|^2\,\sin\theta\,{\rm d\theta}\,{\rm
d}\phi\nonumber\\
&&
+\frac{1}4\int_{0}^\pi\!\int_{0}^{2\pi}
\bigg|\frac{\partial
F}{\partial \phi}(\theta,\phi)\bigg|^2\frac{1}{\sin\theta}{\rm
d\theta}\,{\rm d}\phi+\frac{1}4\int_{0}^\pi\!\int_{0}^{2\pi}
\bigg|\frac{\partial
F}{\partial \theta}(\theta,\phi)\bigg|^2\ {\sin\theta}\,{\rm d\theta}\,{\rm
d}\phi \nonumber \\
&&
+\frac1{16}\int_{0}^\pi\!\int_{0}^{2\pi}
|F(\theta,\phi)|^2 {\sin\theta}
\, {\rm d\theta}\,{\rm d}\phi.
\label{eq:normH2}
\end{eqnarray}
This property follows from the fact that $\{(n+1/2)^{-2}Y_n^m : n = 0,1, 2,
\ldots,  |m| \leq n\}$ is
orthonormal with respect to this inner product.

\subsection{Sobolev-like  Fourier coefficient spaces}

For $s \geq 0$, we introduce a decomposition of ${\cal H}^s$ as an orthogonal
direct sum. Such a decomposition, roughly speaking, comes from interchanging the
 order in which the series in \eqref{eq:series} is summed. In order to simplify
the exposition, we start first with $F\in {\cal H}^0$.
Hence
\begin{eqnarray}
 F(\theta,\phi)& =&\sum_{n=0}^\infty \sum_{m=-n} ^n \widehat{F}_{n,m}
Y_n^m(\theta,\phi) =
\sum_{m=-\infty}^\infty \bigg(\sum_{n=|m|}^\infty \widehat{F}_{n,m}
Q_{n}^m(\theta)\bigg) e_m(\phi)\nonumber\\
&=:&
\sum_{m=-\infty}^\infty  ({\cal F}_m F)(\theta)\,  e_m(\phi).\label{eq:SeriesFm}
\end{eqnarray}
The function  $F(\theta,\,\cdot\,)$ is almost everywhere a function in
$L^2(0,2\pi)$ and  as such  it can be expanded in its Fourier series. By
matching the Fourier coefficients of the same order, we conclude that ${\cal
F}_m F$ admits the integral expression
\begin{equation}\label{eq:Fm}
 {\cal F}_m F=
\int_0^{2\pi} F(\,\cdot\,,\phi)\overline{e_{m}(\phi)}\,{\rm d}\phi=
\int_0^{2\pi} F(\,\cdot\,,\phi) {e_{-m}(\phi)}\,{\rm d}\phi.
\end{equation}
Furthermore, 
\begin{equation}
 \label{eq:Norm0F}
\|F\|_{{\cal H}^0}^2= \int_0^\pi\bigg[\int_0^{2\pi} |F(\theta,\phi)|^2\,{\rm
d}\phi\bigg]\,\sin\theta\,{\rm d}\theta= \sum_{m=-\infty}^\infty \|{\cal F}_m
F\|^2_{L^2_{\sin}}
\end{equation}
with
\[
\|f\|_{L^2_{\sin}}:= \bigg[\int_0^\pi|f(\theta)|^2\sin\theta\,{\rm
d}\theta\bigg]^{1/2}.
\]
Then, if  $L^2_{\sin}$ is the corresponding  $L^2 $ weighted space,
we obtain
\begin{equation}\label{eq:sinL2}
 {\cal H}^0= \bigoplus_{m=-\infty}^\infty \big\{f\otimes e_m\ \big|\
f|_{(0,\pi)}\in
L^2_{\sin},\ f(\cdot+2\pi)=f, \ f(-\, \cdot\,)=(-1)^m f\big\},
\end{equation}
where
\[
 (f\otimes e_m)(\theta,\phi):=f(\theta)e_m(\phi)=\frac{1}{\sqrt{2\pi}}
f(\theta)\exp({\rm i}m\phi).
\]
Observe that the sum above is actually
orthogonal.

We extend this decomposition to ${\cal H}^s$    by introducing
new  spaces, one for each Fourier mode. Let 
\begin{equation}
\label{eq:def_Wmr}
 W_m^s:=\Big\{f:\R\to\C\ \Big|\ \ f\otimes
e_m\in {\cal H}^s\Big\},
\end{equation}
endowed with the image norm
\begin{equation}\label{eq:defNormWms}
\|f\|_{W_m^s}:=\|f\otimes e_m\|_{{\cal H}^s}.
\end{equation}
Clearly,  $W_m^s=W_{-m}^s$.
We note that the definition  of $W_m^s$ imposes naturally  some periodicity and
parity conditions on its elements, namely
\begin{equation}
\label{eq:periodicity-conditions}
f \in W_m^s  \quad \Longrightarrow\quad f(\,\cdot\,+ 2\pi)=f,\quad
f(-\,\cdot\,)=(-1)^m f
\end{equation}
which are simply reflections on those satisfied by the elements of ${\cal
H}^s$
(see \eqref{eq:Fstar1}).

It is easy to see that, for each $m \in \mathbb{Z}$,  $W_m^s$ is a Hilbert
space 
and  a  complete orthonormal set in $W_m^s$ is 
 $\big\{(n+1/2)^{-s}
Q_n^m : n \geq |m|\big\}$. Hence,
\begin{equation}
\label{eq:normWmr}  f = \sum_{n=|m|}^\infty
 \widehat{f}_m(n) Q_n^m,\qquad
 \|f\|_{W_m^s}=\bigg(\sum_{n=|m|}^\infty
(n+{\textstyle\frac12})^{2s}|\widehat{f}_m(n)|^2\bigg)^{1/2}
\end{equation}
with
\begin{eqnarray*}
 \widehat{f}_m(n)&:=&\int_{0}^\pi f(\theta)Q_n^{m}(\theta)\sin\theta\,{\rm
d}\theta=\int_0^\pi \int_0^{2\pi} \big(f\otimes e_m \big)(\theta,\phi)
\overline{Y_n^m(\theta,\phi)}\,{\rm d}\theta\,{\rm
d}\phi\\
&=&\widehat{(f\otimes e_m)}_{n,m}.
\end{eqnarray*}
Clearly, $ {\cal F}_m:{\cal H}^s\to W_m^s$ is continuous and onto. Moreover, 
using \eqref{eq:SeriesFm}--\eqref{eq:Norm0F}, 
\begin{equation}
 \label{eq:Hr_in_Wmr}
 F=\sum_{m=-\infty}^\infty ({\cal F}_m F)\otimes e_m ,\qquad
 \|F\|_{{\cal H}^s}^2 =\sum_{m=-\infty}^\infty \|{\cal F}_m F\|_{W_m^s}^2,
\end{equation}
with the first series converging in ${\cal H}^s$. 
That is, we obtain an orthogonal direct sum decomposition
of ${\cal
H}^s$:
\[
 {\cal H}^s =\bigoplus_{m=-\infty}^\infty \{f\otimes e_m\:|\: f\in W_m^s\}.
\]

\subsection{Properties of Fourier coefficient
spaces}\label{sec:prop_fourier_spaces}
We  derive some  properties of $W_m^s$  that we will use in
this article. Proofs of the results in this section are given in
 Appendix~\ref{sec:fourier_space_appendix}. 
For further study of such spaces, we refer
to~\cite{DoHeSa:2009} where a closely related family of spaces, namely that
obtained by making the change of variables $x=\cos\theta$, are studied.

It is easy to see that for $r> s$ the injection $W_m^r\subset W_m^s$ is
compact. Moreover, because of  \eqref{eq:normWmr},
\begin{equation}
 \label{eq:inequatityWms}
\|f\|_{W_m^s}\le   (|m|+{\textstyle\frac12})^{s-r} \|f\|_{W_m^r},\qquad \forall
r\ge s.
\end{equation}

Equations \eqref{eq:Hr_in_Wmr} and \eqref{eq:inequatityWms} prove that for
$r\ge s$
\begin{equation}
 \label{eq:inequatityHs}
\sum_{m=-\infty}^\infty ( |m|+{\textstyle\frac12})^{2(r-s)}\|{\cal F}_m
F\|_{W_m^s}^2\le
\|F\|_{{\cal H}^r}^2.
\end{equation}
Hence,
\begin{eqnarray}
\Big\|F-\sum_{ -N<m\le N}\big({\cal F}_mF\big) \otimes e_m\Big\|_{{\cal H}^s}&=&
\bigg[\Big(\sum_{m=N+1}^\infty +\sum_{m=-\infty}^{-N}\Big)\big\|{\cal
F}_mF\big\|^2_{W_m^s}\bigg]^{1/2}\nonumber\\
&\le &(N+{\textstyle\frac12})^{s-r} \Big[\sum_{|m|\ge 
N}(|m|+{\textstyle\frac12})^{2(r-s)}\big\|{\cal
F}_mF\big\|^2_{W_m^s}\Big]^{1/2}\nonumber\\
&\le& (N+{\textstyle\frac12})^{s-r}\|F\|_{{\cal H}^r}
\label{eq:I-S}
\end{eqnarray}
for all $r\ge s$.

\begin{proposition}\label{prop:eqWm1Y1}
For all $m\in\Z$,
\begin{eqnarray}
  \|f\|_{W^1_m}^2&=& \frac{1}{4}\int_0^{\pi}
|f(\theta)|^2\sin\theta\,{\rm d}\theta+  m^2\int_0^\pi |f(\theta)|^2\frac{{\rm
d}\theta}{\sin\theta} +\int_0^\pi
|f'(\theta)|^2\,\sin\theta\,{\rm d}\theta\label{eq:01:theo:eqWm1Y1}.
\end{eqnarray}
Moreover, if $|m|\ge 2$,
\begin{eqnarray}
  \|f\|_{W^2_m}^2&=&\frac{1}{16}\int_0^\pi
|f(\theta)|^2\sin\theta\,{\rm d}\theta+
\frac{9m^2}{4}\int_0^\pi
|f(\theta)|^2\frac{{\rm d}\theta}{\sin\theta}+
(m^4-4m^2)\int_0^\pi
|f(\theta)|^2\frac{{\rm d}\theta}{\sin^3\theta}\nonumber\\
&&+\frac{1}{4}\int_0^\pi |f'(\theta)|^2\sin\theta\,{\rm
d}\theta+
(1+2m^2)\int_0^\pi |f'(\theta)|^2\frac{{\rm d}\theta}{\sin\theta}+\int_0^\pi
|f''(\theta)|^2\,{\sin \theta}\,{\rm d}\theta. \nonumber \\
\label{eq:02:theo:eqWm1Y1}
\end{eqnarray}
\end{proposition}

The norms \eqref{eq:01:theo:eqWm1Y1}-\eqref{eq:02:theo:eqWm1Y1} contain three
and six integral terms respectively. We can, however, define equivalent
norms with fewer
terms. Let  
\begin{eqnarray}
\|f\|_{{Z_m^1}}^2&:=& m^2 \int_0^{\pi
}|f(\theta)|^2\frac{\rm
d\theta}{\sin\theta}+
 \int_0^{\pi }|f'(\theta)|^2{\sin\theta}\,{\rm d\theta}\label{eq:Ym}\\
\|f\|_{{Z_m^2}}^2&:=& m^4 \int_0^{\pi }|f(\theta)|^2\frac{\rm
d\theta}{\sin^3\theta}+
m^2\int_0^{\pi }|f'(\theta)|^2\frac{\rm
d\theta}{\sin\theta}+\int_0^{\pi }|f''(\theta)|^2{\sin\theta}\,{\rm
d\theta}.\label{eq:Zm}
\end{eqnarray}
The following result shows that, for $i = 1, 2$,  the spaces $W_m^i$
and $Z_m^i$ are equivalent for all $|m| \geq i$ with
equivalence constants independent of $m$.

\begin{corollary}\label{cor:equivNorms}
For all $m\in\Z$ with  $|m|\ge 1$,
\begin{equation}\label{eq:01:cor:equivNorms}
\|f\|_{W_0^1}\le  \|f\|_{W_m^1}\le \frac{\sqrt{5}}2 \|f\|_{{Z_m^1}}\le
\frac{\sqrt{5}}2\|f\|_{W_m^1}. 
 \end{equation}
Moreover, for all  $m\in\Z$ with $|m|\ge 2$,
\begin{equation}\label{eq:02:cor:equivNorms}
\frac{1}{\sqrt{3}}\|f\|_{W_m^2}\le \|f\|_{{Z_m^2}}\le  \sqrt{6} \|f\|_{W_m^2}.
 \end{equation}
\end{corollary}

%
%
%
%




\begin{remark} \label{remark:continuityWm}
Notice that  if $f\in W_m^1\cong Z_m^1$ for $m\ne 0$, then $f^2$ and its
derivative are in $L^1(0,\pi)$. Thus, from the Sobolev
embedding theorem $|f|^2$ is continuous, and so is $f$. Moreover, $f$ vanishes
at the end points $\{0,\pi\}$. 
This regularity property  need not be true for functions in $W_0^1$. For
instance, $\big|\log|\sin\theta|\big|^{1/2}\in W_0^1$ but this function
is obviously not continuous. 
\end{remark}
We complete this section by studying the regularity of
$W_m^s$ from a classical
Sobolev point of view.
Let us introduce the $2\pi$-periodic Sobolev spaces
 \begin{equation}\label{eq:H_hash}
  H_{\#}^r:=\Big\{f\in H_{\rm loc}^r(\R)\ \Big|  \ f=f(\cdot+2\pi)\Big\}
 \end{equation}
endowed with the norm
\begin{equation}
 \label{eq:SobolevNorm}
 \|f\|_{  H_{\#}^r}^2:=|\widehat{f}(0)|^2+\sum_{m\ne 0}| m|^{2r}
|\widehat{f}(m)|^2,\qquad
\widehat{f}(m)= \frac1{2\pi}\int_0^{2\pi} f(\theta)\exp(-{\rm i}m\theta)\,{\rm
d}\theta.
\end{equation} For $r=0$, $\|\cdot\|_{  H_{\#}^0}$
is simply the $L^2(0,2\pi)$ norm, and for non-negative integer values of
$r$, an equivalent norm is
given by
\begin{equation}
 \label{eq:SobolevNorm:2}
\bigg[
\int_{0}^{2\pi}|f(\theta)|^2\,{\rm d}\theta+
\int_{0}^{2\pi}|f^{(r)}(\theta)|^2\,{\rm d}\theta\bigg]^{1/2}.
\end{equation}


 \begin{proposition}
 \label{prop:inclusion_in_Hr}
 For all $r>0$ there exists $C_r>0$ independent of $m$ and $f$ such that
 \begin{equation}
 \label{eq:01:prop:inclusion_in_Hr}
  \|f\|_{H^r_\#}\le C_r\|f\|_{W_{m}^{r+1/2}},\qquad \forall f\in
W_m^{r+1/2}.
 \end{equation}
Further,
 \begin{eqnarray}
 \label{eq:02:prop:inclusion_in_Hr}
  \|f\|_{W_m^0}&\le& \|f\|_{H^0_\#},
\label{eq:03a:prop:inclusion_in_Hr}\\
 \|f\|_{W_m^1}&\le&
C(1+|m|) \|f\|_{H^1_\#}
 \label{eq:03b:prop:inclusion_in_Hr},
 \end{eqnarray}
with $C$ independent of $f$ and $m$.
 \end{proposition}

\begin{proposition}\label{prop:ineqY1} There exists $C>0$  such that for all
$f\in W_0^1$
 \begin{eqnarray*}
  \int_0^\pi |f(\theta)|^2\,{\rm d}\theta&\le& C\bigg[
\int_0^\pi |f(\theta)|^2\sin \theta \,{\rm d}\theta\\
&&\ \ +
\bigg(
\int_0^\pi |f(\theta)|^2\sin \theta \,{\rm d}\theta\bigg)^{1/2}\bigg(
\int_0^\pi |f'(\theta)|^2\sin \theta \,{\rm d}\theta\bigg)^{1/2}
\bigg].
  \end{eqnarray*}
\end{proposition}

\section{Optimal order convergence analysis}\label{sec:optim_conv}

In this section, we develop several results and hence prove the
following optimal order convergence property of the interpolatory
operator ${\cal Q}^{\rm gl}_N$ on the sphere. 
\begin{theorem}\label{theo:main}
Let $s = 0, 1$.
For all $r>s+3/2$ there exists $C_r>0$ such that for all $F\in{\cal H}^r$ 
the following optimal order estimate holds:
\begin{equation}\label{eq:opt_conv}
(1-s) \|{\cal Q}^{\rm gl}_N F- F\|_{{\cal H}^0 }+sN
 \|{\cal Q}^{\rm gl}_N F- F\|_{{\cal H}^1 }
\le C_r N^{-r}\|F\|_{{\cal
H}^r }.
\end{equation}
\end{theorem}
We demonstrate the optimal
order convergence property with several numerical experiments in
Section~\ref{sec:num_exp}.  

The discrete spaces and the functional frame where the analysis will be carried
out have been
already introduced in previous sections. The proof of Theorem~\ref{theo:main}
will rely on using 
certain  semi-discrete operators $\rho_N^m$  and two interpolation
operators ${\rm q}_N^{\rm e}$ and ${\rm q}_N^{\rm o}$  in the
polar angle $\theta$.  These operators facilitate a representation of the  
interpolation error as a sum of that introduced in the approximation of the
central $2N$
Fourier
coefficients of $F$ plus the error arising from ignoring  the tail
of the Fourier series \eqref{eq:Hr_in_Wmr}.

\subsection{Fourier analysis}\label{sec:four_anal}
Recalling the definition of ${\cal F}_m$ in~\eqref{eq:SeriesFm},
the aim of this part is to investigate properties of the
semi-discrete maps
\begin{equation}\label{eq:rho}
\left(\rho_N^m F\right)(\theta):=\sum_{\ell=-\infty}^\infty  \big({\cal
F}_{m+2\ell
N}F\big)(\theta),\qquad
 m=-N+1,\ldots,N.
\end{equation}
The following bound 
will be used repeatedly in this article: For $r>1$
and for all $|m|\le N$
\begin{equation}
\label{eq:bound}
 \sum_{\ell\ne 0} |m+2\ell N|^{-r}\le C_{r} N^{-r},\qquad
C_r:= \sum_{\ell=1}^\infty \frac{1}{\ell^r}<\infty.
\end{equation}

\begin{lemma}\label{lemma:UnifConvergence}
For all $r>1$ there exists $C_r>0$ such that  for any $F\in{\cal H}^r$
\[
\sum_{m=-\infty}^\infty \sum_{n=|m|}^\infty
|\widehat{F}_{n,m} Y_n^{m}|=
\frac{1}{\sqrt{2\pi}}\sum_{m=-\infty}^\infty \sum_{n=|m|}^\infty
|\widehat{F}_{n,m} Q_n^{m}|
 \le C_r \|F\|_{{\cal H}^r}.
\]
Moreover,
\[
 F(0,\,\cdot\,)= \big({\cal F}_0 F\big)(0),\quad
 F(\pi,\,\cdot\,)=\big( {\cal F}_0 F\big)(\pi), 
\]
and therefore  $F\in{\cal C}$.
\end{lemma}
\begin{proof}
Since $F \in {\cal H}^{1+\varepsilon}$, for some $\varepsilon > 0$,
using 
\eqref{eq:defHarmonicSpherics},
\begin{eqnarray*}
\sum_{m=-\infty}^\infty \sum_{n=|m|}^\infty
|\widehat{F}_{n,m} Y_n^{m}(\theta,\phi)|&=&\frac{1}{\sqrt{2\pi}}
\sum_{m=-\infty}^\infty\sum_{n=|m|}^{\infty}
|\widehat{F}_{n,m} Q_n^{m}(\theta)| \\
&& \hspace{-4.5cm} \le\
\frac{1}{\sqrt{2\pi}} \bigg(\sum_{m=-\infty}^\infty\sum_{n=|m|}^{\infty}
|\widehat{F}_{n, m
}|^2\big(n+{\textstyle\frac{1}2}\big)^{2+2\varepsilon}\bigg)^{1/2}\!
\bigg(\sum_{n=0}^{\infty}  \big(n+{\textstyle\frac{1}2}\big)^{-2-2\varepsilon}
\sum_{m=-n}^n
|Q_n^{m}(\theta)|^2\bigg)^{1/2}\\
&& \hspace{-4.5cm} =\ \frac{1}{\sqrt{2\pi}}
\bigg[\sum_{n=0}^{\infty}  \sum_{m=-n}^n
|\widehat{F}_{n,m}|^2\big(n+{\textstyle\frac{1}2}\big)^{2+2\varepsilon}\bigg]^{
1/2 }
 \bigg[  \sum_{n=0}^{\infty} 
(n+\textstyle\frac12)^{-1-2\varepsilon}
\bigg]^{1/2} \le  C_\varepsilon\|F\|_{{\cal H}^{1+\varepsilon}}.
\end{eqnarray*}
In the penultimate step, we have used the addition theorem~\cite[Theorem
2.4.5]{Ne:2001}
\[
 \sum_{m=-n}^n
|Q_n^m(\theta)|^2= n+{\textstyle\frac{1}2} ,\qquad \forall\theta.
\]
In particular, $F$ is continuous. Moreover, for $m\ne 0$, 
$Q_n^m(0)=Q_n^m(\pi)=0$ for all $n$. Hence using~\eqref{eq:SeriesFm}, 
\[
F(0,\phi)=\sum_{n=0}^{\infty}
 \widehat{F}_{n,0}  Q_n^{0}(0)=\big({\cal F}_0 F\big)(0),\qquad \forall \phi.
\]
Analogously, it can be shown that $F(\pi,\,\cdot\,)=\big({\cal F}_0
F\big)(\pi)$.
Thus $F\in{\cal C}$. 
\end{proof}
 \begin{lemma}\label{lemma:pNm}
 For all $|m|\le N$ and $r>1$, $\rho_N^mF\in{\cal C}[0,\pi]$ and
 \begin{subequations}\label{eq:02:lemma:pNm}
 \begin{eqnarray}
 \big(\rho_N^mF\big)(0)\!\!&=&\!\!\big(\rho_N^mF\big)(\pi)=0,\qquad \text{if
}m\ne
 0,\label{eq:02a:lemma:pNm}\\
 \big(\rho_N^0F\big)(0)\!\!&=&\!\!\big({\cal F}_0 F\big)(0)=F(0,\phi),\ \
 \big(\rho_N^0F\big)(\pi)\:=\:\big({\cal F}_0 F\big)(\pi)=F(\pi,\phi),\ \ 
 \forall
 \phi\in\R.\qquad\label{eq:02b:lemma:pNm}
 \end{eqnarray}
 \end{subequations}
 \end{lemma}
 \begin{proof} 
Follows  from the definition of ${\cal F}_m$ in~\eqref{eq:SeriesFm} and 
Lemma~\ref{lemma:UnifConvergence}.
%
  \end{proof}

Given a general interpolation operator ${\cal Q}_N:{\cal C}\to\chi_N$
fulfilling the conditions stated in
\eqref{eq:nodesInterpolationAbstr}, we consider the associated one
dimensional interpolants: For $f\in {\cal C}[0,\pi]$ find ${\rm q}_{N}^{\rm
 e} f \in \mathbb{D}_N^{\rm  e }$ and ${\rm q}_{N}^{\rm
 o} f \in \mathbb{D}_{N-2}^{\rm  o }$  such that 
\begin{subequations}
 \label{eq:qN}
\begin{eqnarray}
{\rm q}_N^{\rm  e }
f(\xi_i)&=&f(\xi_i),\quad i=0,\ldots,N,\label{eq:q_Ne}\\
 {\rm q}_N^{\rm  o }
f(\xi_j)&=&f(\xi_j),\quad j=1,\ldots,N-1.\label{eq:q_No}
\end{eqnarray}
\end{subequations}
With the help of these interpolants we are able to write  ${\cal Q}_NF$ in terms
of the
Fourier coefficients of $F$.

\begin{lemma}\label{lemma:QN:Fourier}
For all $F\in{\cal H}^r$ with $r>1$
\begin{equation}\label{eq:interp_new_rep}
\big({\cal Q}_N F\big)(\theta,\phi) =\sum_{-N+1\le  m\le
N\atop \text{even }m} \big({\rm q}_N^{\rm e} \rho_N^m
F\big)(\theta)e_{m}(\phi)+
\sum_{-N+1\le  m\le
N\atop \text{odd }m} \big({\rm q}_N^{\rm o} \rho_N^m
F\big)(\theta)e_{m}(\phi).
\end{equation}
\end{lemma}
\begin{proof}  Note that, because of \eqref{eq:02:lemma:pNm}, for even $m\ne 0$,
$({\rm
q}_N^{\rm e} \rho_N^{m})(0)=({\rm q}_N^{\rm e} \rho_N^{m})(\pi)=0$ which proves
that the element of the right hand side of~\eqref{eq:interp_new_rep}  is
contained in
$\chi_N$.  We then just
have to prove that  the right hand side of~\eqref{eq:interp_new_rep} satisfies
the interpolation
conditions~\eqref{eq:nodesInterpolationAbstr}.

Since
$F\in{\cal H}^r$, with $r>1$, Lemma \ref{lemma:UnifConvergence} 
shows that the Laplace series of $F$ converges absolutely. 
Hence the following manipulations are fully  justified. For $k=-N+1,\ldots,N$,
\begin{eqnarray*}
F(\theta,{\textstyle\frac{k\pi}{2N}})&=&\sum_{m=-\infty}
^\infty
\big({\cal F}_{m}F\big)(\theta)
e_{m}\big({\textstyle\frac{k\pi}{2N}}\big)=\sum_{-N+1\le  m\le
N}\sum_{\ell=-\infty}
^\infty
\big({\cal F}_{m+2\ell N }F\big)(\theta)
e_{m+2\ell N }\big({\textstyle\frac{k\pi}{2N}}\big)
\\
&=&\sum_{-N+1\le  m\le
N}\bigg(\sum_{\ell=-\infty} ^\infty
\big({\cal F}_{m+2\ell N }F\big)(\theta)
\bigg)e_{m}\big({\textstyle\frac{k\pi}{2N}}\big)\\
&=&\sum_{-N+1\le  m\le
N} (\rho_N^m F)(\theta)e_{m}\big({\textstyle\frac{k\pi}{2N}}\big).
\end{eqnarray*}

Therefore, for $j=0,\ldots, N$  and $k=-N+1,\ldots,N$,
\begin{eqnarray*}
F({\textstyle \xi_j},{\textstyle\frac{k\pi}{2N}})&=&\sum_{-N+1\le  m\le
N}
(\rho_N^m F)(\xi_j)e_{m}\big({\textstyle\frac{k\pi}{2N}}\big)\\
&=&\sum_{-N+1\le  m\le
N\atop \text{even }m} \big({\rm q}_N^{\rm e} \rho_N^m
F\big)(\xi_j)e_{m}\big({\textstyle\frac{k\pi}{2N}}\big)+
\sum_{-N+1\le  m\le
N\atop \text{odd }m} \big({\rm q}_N^{\rm o} \rho_N^m
F\big)(\xi_j)e_{m}\big({\textstyle\frac{k\pi}{2N}}\big).
\end{eqnarray*}
 Thus we obtain the representation~\eqref{eq:interp_new_rep} 
\end{proof}


Denote by
\begin{equation}\label{eq:qGL}
{\rm q}^{\rm e}_{N,\rm gl}:{\cal C}[0,\pi]\to \mathbb{D}_N^{\rm e},
\quad \text{and}\quad
{\rm q}^{\rm o}_{N,\rm gl}:{\cal C}[0,\pi]\to \mathbb{D}_{N-2}^{\rm o}
\end{equation}
respectively the interpolant projections defined in
\eqref{eq:qN}  with nodes  $\{\theta_j\}_{j=0}^N$, the
Gauss-Lobatto points. Then
 using
Lemma
\ref{lemma:QN:Fourier},
\begin{equation}\label{eq01:sec42}
{\cal Q}^{\rm gl}_N F(\theta,\phi) =\sum_{-N+1\le  m\le
N\atop \text{even }m} \big({\rm q}^{\rm e}_{N,\rm gl}\, \rho_N^m
F\big)(\theta)e_{m}(\phi)+
\sum_{-N+1\le  m\le
N\atop \text{odd }m} \big({\rm q}^{\rm o}_{N,\rm gl}\,\rho_N^m
F\big)(\theta)e_{m}(\phi).
\end{equation}
For  any $m \in \mathbb{Z}$, we introduce a simpler notation 
\begin{equation}\label{eq:qNm}
 {\rm q}_{N,\rm gl}^m:=\left\{ \begin{array}{ll}
                       {\rm q}_{N,\rm gl}^{\rm e},&\text{if $m$ is even},\\
                       {\rm q}_{N,\rm gl}^{\rm o},&\text{if $m$ is odd}.
                      \end{array}\right.
\end{equation}

We recall the spaces $Z_m^i$, for $i = 1,2$, induced respectively by the
norms~\eqref{eq:Ym} and~\eqref{eq:Zm} and that for $|m| \geq
i$, these are equivalent to  the spaces
$W_m^i$ in~\eqref{eq:def_Wmr} with respective norms given
by~\eqref{eq:01:theo:eqWm1Y1}  and~\eqref{eq:02:theo:eqWm1Y1}.
The next two results give  properties of ${\rm q}_{N,\rm gl}^m$ when
applied to elements in these spaces. 
Derivation of  these results  require  new technical tools and hence 
we postpone the proofs of the following two results  to
Appendix~\ref{sec:appB_fourier_anal_proof}.

\begin{lemma}\label{lemma:pL2}
There exists $C>0$ such that for all $f\in Z_m^1 $ with $m\ne 0$ and $N\ge 2$
\begin{equation}
\label{eq01:lemma:pL2}
\|{\rm q}^{m}_{N,\rm gl}f\|_{L^2_{\sin}}\le C\Big[
 \|f\|_{L^2_{\sin}}  + \frac{1}N \|f\|_{Z_m^1} \Big]. 
\end{equation}
Besides, for all $f\in Z_m^2$ with $m\ge 2$
\begin{equation}
\label{eq02:lemma:pL2}
\|{\rm q}^{m}_{N,\rm gl}f\|_{Z_m^1}\le C\Big[
\|f\|_{{Z_m^1}} + \frac{1}N \|f\|_{{Z_m^2}} \Big],
\end{equation}
where $C$ is independent of $m$ and $N$.
\end{lemma}
\begin{proof}
 See Lemma \ref{lemma:pL2prime} and
Corollary \ref{cor:pL2prime2}. 
\end{proof}

\begin{proposition}\label{prop:convqNeo} 
Let $s= 0, 1$. Then, for
$s=0$ and $r>1$ and  for $s=1$ and  $r\ge 2$, there exists
$C_r>0$ such that for any  $N\ge 2$  and $f\in W_m^r$
\[
\|{\rm q}^{m}_{N,\rm gl}f-f\|_{W_{m}^s}\le C_r  
N^{s-r}\|f\|_{W_{m}^r}.
\]
Moreover, for $m\ne 0$ and $s=0$ the result holds also for $r=1$.
\end{proposition}
\begin{proof}
 See Propositions \ref{prop:conv01} and \ref{prop:conv02}.
\end{proof}


The main result of this section is stated and proven below. 
\begin{theorem}
\label{theo:main01}
Let $s = 0, 1$. For $r>3/2+s$, there exists
$C_r>0$ such that for all $N\ge 2$  and   $|m|\le N$,
\begin{eqnarray}
 \|{\rm q}^{m }_{N,\rm gl} \rho_N^m F -{\cal F}_m
F\|_{W_m^ s }\!\!&\le&\!\! C_r   N^{ -r}
\bigg[\sum_{\ell=-\infty}^\infty \|{\cal
F}_{m+2\ell N }F\|^2_{W_{m+2\ell N }^{r}}\bigg]^{1/2}\le C_r   N^{ s-r}
\|F\|_{{\cal
H}^r}. \quad \label{eq:02:theo:main01}
\end{eqnarray} 
\end{theorem}
\begin{proof} 
We have
\begin{eqnarray*}
\|{\rm q}_{N,\rm gl}^{m} \rho_N^m F -{\cal F}_m
F\|_{W_m^0}
&=&\|{\rm q}_N^m \rho_N^m F -{\cal F}_m
F\|_{L^2_{\sin}}\\
&\le&   \|{\rm q}_N^m (\rho_N^m F-{\cal F}_m
F)\|_{L^2_{\sin}} +
 \| {\rm q}_N^m {\cal F}_mF-  {\cal F}_{m}F  \|_{L^2_{\sin}} \\
&=:&\|E_1\|_{L^2_{\sin}} +\|E_2\|_{L^2_{\sin}}.
\end{eqnarray*}
For the first term in the bound, we obtain
\begin{eqnarray}
\|E_1\|^2_{W_m^0}
&\le&
\bigg[\sum_{\ell\ne 0}|m+2\ell N|^{2-2r}
\bigg] \bigg[\sum_{\ell\ne 0} |m+2\ell N|^{2r-2}\| {\rm q}_N^{m}{\cal
F}_{m+2 \ell
N} F \|^2_{L^2_{\sin}}\bigg] \nonumber\\
&\le& C_r N^{2-2r} \bigg[\sum_{\ell\ne 0} |m+2\ell N|^{2r-2} \big(\|{\cal
F}_{m+2 \ell
N} F \|^2_{L^2_{\sin}}+N^{-2}\|{\cal
F}_{m+2 \ell
N} F \|^2_{Z_{m+2\ell N}^1}\big)\bigg]  \nonumber\\
&\le& C_r N^{2-2r} \bigg[\sum_{\ell\ne 0} |m+2\ell N|^{2r-2} \big( |m+2\ell
N|^{-2}+N^{-2}\big)\|{\cal
F}_{m+2 \ell
N} F \|^2_{W_{m+2\ell N}^1}\big)\bigg] \nonumber\\ 
&\le& 2C_r N^{-2r}
\bigg[\sum_{\ell\ne 0}   \|{\cal
F}_{m+2 \ell
N} F \|^2_{W_{m+2\ell N}^r} \bigg], \label{eq:E1:the0:main01}
\end{eqnarray}
where we have applied successively \eqref{eq:bound},
\eqref{eq01:lemma:pL2} of
Lemma \ref{lemma:pL2},
Corollary
\ref{cor:equivNorms} and \eqref{eq:inequatityWms}.
On the other hand, Proposition
\ref{prop:convqNeo} yields
\begin{eqnarray}
\|E_2\|_{W_m^0}&\le& C'N^{-r}\|{\cal F}_{m}F\|_{W_m^r}.
\label{eq:E2:the0:main01}
\end{eqnarray}
 The result for the  case $s=0$  follows now from 
\eqref{eq:E1:the0:main01} and \eqref{eq:E2:the0:main01}.

For $s=1$, we proceed as before, but using now  \eqref{eq02:lemma:pL2} of Lemma
\ref{lemma:pL2}:
 \begin{eqnarray}
\|E_1\|^2_{W_m^1}
&\le&
\bigg[\sum_{\ell\ne 0}|m+2\ell N|^{4-2r}
\bigg] \bigg[\sum_{\ell\ne 0} |m+2\ell N|^{2r-4}\| {\rm q}_N^{m}{\cal
F}_{m+2 \ell
N} F \|^2_{W_m^1 }\bigg] \nonumber\\
&\le& C_r N^{4-2r} \bigg[\sum_{\ell\ne 0} |m+2\ell N|^{2r-4} \big(\|{\cal
F}_{m+2 \ell
N} F \|^2_{Z_{m+2\ell N}^1}+N^{-2}\|{\cal
F}_{m+2 \ell
N} F \|^2_{Z_{m+2\ell N}^2}\big)\bigg]  \nonumber\\
&\le& C_r N^{4-2r} \bigg[\sum_{\ell\ne 0} |m+2\ell N|^{2r-4} \big( |m+2\ell
N|^{-2}+6N^{-2}\big)\|{\cal
F}_{m+2 \ell
N} F \|^2_{W_{m+2\ell N}^2}\big)\bigg] \nonumber\\ 
&\le& 7 C_r N^{2-2r}
\bigg[\sum_{\ell\ne 0}   \|{\cal
F}_{m+2 \ell
N} F \|^2_{W_{m+2\ell N}^r} \bigg]. \label{eq:02:E1:the0:main01}
\end{eqnarray}
Finally, Proposition \ref{prop:convqNeo} yields
\[
\|E_2\|_{W_m^1}\le  C_r  N^{1-r}\|{\cal F}_m F\|_{W_m^r}.
\]
\end{proof}

\subsection{Proof of Theorem \ref{theo:main}}

For $s=0, 1$, using
\eqref{eq01:sec42} and \eqref{eq:Hr_in_Wmr} we obtain
 \begin{eqnarray*}
 \|{\cal Q}^{\rm gl}_N F-F\|_{{\cal H}^s}^2&=& \sum_{-N<m\le N }
\|{\rm q}_{N,\rm gl}^{m} \rho_N^m F -{\cal
F}_m
F\|_{W_m^s}^2 +
\bigg(\sum_{m=-\infty}^{-N} + \sum_{m=N+1}^{\infty}
\bigg)\|{\cal F}_m
F\|_{W_m^s}^2\\
& =:& D_1 + D_2.
 \end{eqnarray*}
The last term can be estimated using \eqref{eq:I-S}
\[
 D_2\le N^{2s-2r}\|F\|_{{\cal H}^r}.
\]
On the other hand, and with the help of  Theorem \ref{theo:main01}, we deduce
that
\begin{eqnarray*}
D_1&\le& C_rN^{2s-2r} \sum_{-N+1\le m\le N} 
\sum_{\ell=-\infty}^\infty\|{\cal
F}_{m+2\ell N }F\|_{W_{m+2\ell N }^{r }}^2
\\
&\le&C_r N^{2s-2r}\sum_{m=-\infty}^\infty\|{\cal F}_ m  F\|_{W_{m}^{r }}^2
=C_r N^{2s -2r}\|F\|_{{\cal H}^{r}}^2.
\end{eqnarray*}
\ \hfill$\Box$

\section{Cubature and analysis  for wideband integrals}\label{sec:cub_anal}
\label{sec:theIntegrals}
In this section, using the tools developed in previous sections,  
we propose  efficient approximations to the class  of  wideband integrals on the
sphere 
defined in~\eqref{eq:gen_int_sph} and analyze the error in such cubature
approximations.

In particular,  for $F \in {\cal{H}}^r$ with $r > 1$ and  $f = F^\circ$ 
in~\eqref{eq:gen_int_sph},  
the wideband integrals (with wavenumber $\kappa \in (0, \infty)$ and 
incident direction  $\ddh =  [0, 0, 1]^T$)  can be written as
\begin{equation}\label{eq:integ}
 I_\kappa (F):=\int_0^\pi\int_{0}^{2\pi} F(\theta,\phi)\exp({\rm
i}\kappa\cos\theta)\,\sin\theta\,{\rm d}\theta\,{\rm d}\phi
=\int_{\sph} F^\circ(\x) \exp({\rm i} \kappa \x \cdot \ddh)  \; {\rm d}\x. 
\end{equation}
We are interested in robust cubature rules that give accurate results 
for small, large and very large values of the wavenumber $\kappa$.

 Following the ideas of Filon rules, or
those more general product integration rules (see for instance
\cite{KrUe:1998}), we propose
\begin{equation}
 I_\kappa (F)\approx \int_0^\pi\int_{0}^{2\pi} \big( {\cal
Q}_N^{\rm gl}F\big) (\theta,\phi)\exp({\rm i}\,{\rm
\kappa}\cos\theta)\,\sin\theta\,{\rm
d}\theta\,{\rm d}\phi=:I_{\kappa,N}^{\rm gl}(F).
\end{equation}

Using~\eqref{eq:Fm} and Lemma~\ref{lemma:QN:Fourier}, we obtain 
for  $F \in {\cal{H}}^r$ with $r > 1$,
\begin{eqnarray}
 I_\kappa (F)&=&\sqrt{2\pi}\int_0^{\pi }\big({\cal F}_0 F\big)(\theta)\exp({\rm
i}\kappa\cos \theta)\,\sin\theta\,{\rm
d}\theta,\label{eq:exIntegral}\\
 I_{\kappa,N}^{\rm gl} (F)&=&\sqrt{2\pi}\int_0^{\pi }\big({\cal F}_0{\cal
Q}_N
F\big)(\theta)\exp({\rm
i}\kappa\cos \theta)\,\sin\theta\,{\rm
d}\theta\nonumber\\
& =&\sqrt{2\pi}\int_0^{\pi }\big({\rm q}_{N, {\rm gl}}^{\rm
e}\rho_N^0F\big)(\theta)\exp({\rm
i}\kappa\cos\theta)\,\sin\theta\,{\rm d}\theta.\quad\quad\label{eq:QuadRule}
\end{eqnarray}
Consequently, 
\[
|I_\kappa(F)-I_{\kappa,N}^{\rm gl}(F)|\le \sqrt{2\pi}\bigg|\int_0^{\pi}
\big({\cal F}_0F
-{\rm
q}_{N,{\rm gl}}^{\rm
e}\rho_N^0F\big)(\theta)\exp({\rm i}\kappa\cos\theta)\,\sin\theta\,{\rm
d}\theta\bigg|.
\]

Before  entering into the analysis of such approximations,
we observe that  computation of ${\rm
q}_N^{\rm e}\rho_N^0F$ does not  require evaluations of 
either  ${\cal F}_{2\ell N} F$ for all $\ell$ or  
${\cal Q}^{\rm gl}_N
F$. 
This is a substantial computational advantage even compared to
using the matrix-free FFT type representation of ${\cal
Q}_N^{\rm gl}F$.  An efficient algorithm to compute the cubature,
based on the proof  of Proposition~\ref{prop:InterpWellDefined}
is as follows.

\paragraph{Algorithm}
\begin{enumerate}
\item Set
\[
 f_0^0=F(0,0),\quad f_0^N=F(\pi,0),\quad f_0^j:=\frac{1}{2N}\sum_{-N+1\le k\le
N}
F({\xi_j,{\textstyle \frac{ k\pi}{N}}}), \quad j = 1, \ldots, N-1.
\]
\item Solve  the one dimensional interpolation
problem at  Gauss-Lobatto based nodes:
\[
\mathbb{D}_N^{\rm e}\ni p_N \quad\mbox{s.t.}\quad p_N(\xi_j)=
 f_0^j,\qquad
j=0,\ldots,N.
\]
(Observe that $p_N=\frac{1}{\sqrt{2\pi}}{\rm q}_{N,{\rm gl}}^{\rm e} \rho_N^0
F$.)
\item Compute
\begin{equation}
 \label{eq:pnIntegral} 2\pi\int_0^{\pi} p_N(\theta)\exp({\rm
i}\kappa\cos\theta)\,\sin\theta{\rm d}\theta.
\end{equation}

\end{enumerate}
Now we require a robust approach for
evaluating~\eqref{eq:pnIntegral}.
After performing the change of variables $
\theta=\arccos x$ we have to evaluate
\[
    \int_{-1}^{1}
p_N(\arccos x)\exp({\rm
i}\kappa x)\, {\rm d}x.
  \]
Since $p_N(\arccos\cdot)$ is now a polynomial, 
there are some accurate, robust, and fast methods in the literature
for evaluating this integral, see~\cite{DoGrSm:2010} and references
therein. In our numerical implementation,
we use the approach developed in~\cite{DoGrSm:2010} for evaluating the one
dimensional integrals~\eqref{eq:pnIntegral}.

Crucial to our wavenumber explicit  error  
analysis  of the cubature rule is Theorem~\ref{theo:main01} which
is essentially the main result for proving optimal order 
error bounds for the interpolatory approximation ${\cal
Q}_N^{\rm gl}F$  in the ${\cal H}^0$
and ${\cal H}^1$ norms. In addition, for the error analysis we
require  additional convergence estimates  in  stronger norms,
defined~\eqref{eq:H_hash}--\eqref{eq:SobolevNorm:2}.

\begin{lemma}
\label{lemma:quadRule}
Let $g\in W_0^{r}$ and $g_N\in \mathbb{D}_N^{\rm e}$. Then for all $s\ge
0$ and $r> s+1/2$
there exists $C_{r}>0$ so that
\begin{equation}
\label{eq:02:lemma:quadRule}
 \|g-g_N\|_{H^s_{\#}}\le C_{r}\big[ N^{s-r+1/2}\|g\|_{W_0^r}+
N^{s+1/2} \|g-g_N\|_{W_0^0}+
N^{s-1/2}\| g-g_N \|_{W_0^1}\big].
\end{equation}
 \end{lemma}
\begin{proof}
We consider the following  $H^0_{\#}$ projection on
$\mathbb{D}_N^{\rm e}$:
\[
{\rm s}_N g(\theta):=\sum_{m=0}^N \widetilde{g}(m)\cos m\theta ,\qquad
\widetilde{g}(m):=\left\{
\begin{array}{ll}
\frac{1}{2\pi}\int_{-\pi}^\pi g(\theta)\,{\rm d}\theta,&\text{if }m=0,\\[2ex]
\frac{1}{\pi}\int_{-\pi}^\pi g(\theta)\cos m\theta \,{\rm d}\theta,\ &
\text{if }m>0.
\end{array}
\right.
\]
Hence for even $g$, as it is in our case, it is easy to prove from
\eqref{eq:SobolevNorm} (see
also \cite{SaVa:2002}) that
\[
\|{\rm s}_N g-g\|_{H^0_{\#}}\le
 N^{-r}\|g\|_{H^r_{\#}}.
\]
The other result  we need in this proof is the well-known and
easy to prove inverse inequality
\[
 \|g_N\|_{H^s_{\#}}\le N^{s-r}\|g_N\|_{H^r_{\#}},\qquad \forall g_N\in {\rm
span}\left\{ \exp({\rm
i}m\,\cdot\,)\ : \  |m|\le  N\right\}, \qquad s \geq r.
\]
Using the above inequalities 
and~\eqref{eq:01:prop:inclusion_in_Hr}, we obtain
\begin{eqnarray*}
 \|g-g_N\|_{H^s_{\#}}&\le& \|g-{\rm s}_Ng\|_{H^s_{\#}}+\|{\rm
s}_Ng-g_N\|_{H^s_{\#}}\\
&\le&  N^{s-r+1/2}\|g\|_{H^{r-1/2}_{\#}}+N^{s}\|{\rm s}_Ng-g_N\|_{H^0_{\#}}\\
&\le&  N^{s-r+1/2}\|g\|_{H^{r-1/2}_{\#}}+N^{s}\big(
\|{\rm s}_Ng-g\|_{H^0_{\#}}+\|g-g_N\|_{H^0_{\#}}\big)\\
&\le& 2 N^{s-r+1/2}\|g\|_{H^{r-1/2}_{\#}}+N^{s}\|g-g_N\|_{H^0_{\#}}\\
&\le& C_{r} N^{s-r+1/2}\|g\|_{W_0^r}+\sqrt{2}N^{s}\|g-g_N\|_{L^2(0,\pi)}.
\end{eqnarray*}
We next apply  Proposition
\ref{prop:ineqY1} to the last term and obtain
\begin{eqnarray}
\|g-g_N\|_{L^2(0,\pi)} & \le&
C\big[\|g-g_N\|_{L^2_{\sin}}^2+ \|g-g_N\|_{L^2_{\sin}}
\|(g-g_N)'\|_{L^2_{\sin}}\big]^{1/2}\nonumber\\
&\le&  C\big(1+N^{1/2}\big)
\|g-g_N\|_{L^2_{\sin}}+
N^{-1/2}\|(g-g_N)'\|_{L^2_{\sin}}.
\end{eqnarray}
The result~\eqref{eq:02:lemma:quadRule} now follows  
by applying Proposition \ref{prop:eqWm1Y1}  to the right hand side above.
\end{proof}

Now we are ready to prove the wavenumber explicit error bounds for the 
cubature approximations $I_{\kappa,N}^{\rm gl} (F)$ of the wideband integrals   
$I_\kappa (F)$. Such error bounds provide robust error estimates for
convergence with respect to the number of cubature points and the wavenumber.
\begin{theorem}\label{theo:convQuadRule}
For $\ell=0,1$ and $r>3/2$ or $\ell=2$ and
$r>4$  there exists $C_r$  such that for all $F \in {\cal H}^r$ ,
\begin{equation}\label{eq:02:theo:convQuadRule}
  | I_\kappa (F)- I_{\kappa,N}^{\rm gl} (F)|   \le C_r \kappa^{-\ell}
N^{\eta(\ell)-r}\|F\|_{{\cal H}^r},
\end{equation}
where
\[
 \eta(\ell):=\left\{
\begin{array}{ll}
0,\quad &\text{if $\ell=0$},\\
 3/2,\quad&\text{if $\ell=1$},\\
4,\quad&\text{if $\ell=2$}.
\end{array}
\right.
\]
\end{theorem}
\begin{proof}
Let
\[
 r_N:=\sqrt{2\pi}\big({\rm q}_{N, {\rm gl}}^{\rm e}\rho_N^0 F- {\cal F}_0
F\big).
\]
For $F$ smooth enough (precise information about the regularity requirements on
$F$ will be given below),   $r_N'(0)=r'_N(\pi)=0$ and therefore the function
\[
 \varphi_N(\theta):=\left\{\begin{array}{rl}
\displaystyle  \frac{r'_N(\theta)}{\sin\theta} ,\ &\text{if
}\theta\in(0,\pi), \\[1.75ex]
r_N''(0),&\text{if }\theta=0,\\[1.25ex]
-r_N''(\pi),&\text{if }\theta=\pi,
                     \end{array}
\right.
\]
is continuous. Hence, using integration by parts we derive
\begin{eqnarray}
 | I_\kappa (F)- I_{\kappa,N} (F)|&=&\bigg| \int_{0}^\pi r_N(\theta)\,\exp({\rm
i}\kappa\cos\theta) \sin \theta\, {\rm d}\theta\bigg|\label{eq:quadRule:01}\\
&=&
\bigg|\frac{\rm 1}{{\rm i}\kappa}\int_0^\pi r_N'(\theta)\,\exp({\rm
i}\kappa\cos\theta)\,{\rm
d}\theta \bigg|\label{eq:quadRule:02}\\
&=&\bigg|\frac{\rm 1}{{\rm i}\kappa}\int_0^\pi
\varphi_N(\theta)\,\exp({\rm
i}\kappa\cos\theta)\,\sin\theta\,{\rm d}\theta\bigg|\nonumber\\
& =& \bigg| -\frac{1}{\kappa^2}
 \varphi_N(\theta)\exp({\rm
i}\kappa\cos\theta)  \Big|_{\theta=0}^{\theta=\pi}+\frac{1}{\kappa^2}
\int_0^\pi \varphi_N'(\theta)\exp({\rm
i}\kappa\cos\theta)\, {\rm d}\theta\bigg|. \nonumber \\
\qquad\label{eq:quadRule:03}
\end{eqnarray}
Hence,   \eqref{eq:quadRule:01} implies that
\begin{equation}
\label{eq:quadRule:k0}
  | I_\kappa (F)- I_{\kappa,N} (F)|\le
\|r_N\|_{L^2_{\sin}}
\end{equation}
whereas  \eqref{eq:quadRule:02} yields
\begin{equation}
 | I_\kappa (F)- I_{\kappa,N} (F)|\le \kappa^{-1} \|r_N'\|_{L^2(0,\pi)}\le
\kappa^{-1}
\|r_N\|_{H^1_{\#} }.
\label{eq:quadRule:k1}
\end{equation}
Finally, \eqref{eq:quadRule:03} shows that
\[
 | I_\kappa (F)- I_{\kappa,N} (F)|\le
\kappa^{-2}\Big[|\varphi_N(0)|+|\varphi_N(\pi)|+\int_0^\pi|\varphi'_N(\theta)|\,
{\rm
d}\theta\Big ].
\]
Proceeding as in the proof of Theorem 2.2 in \cite{DoGrSm:2010}, we obtain
\[
 |\varphi_N(0)|+|\varphi(\pi)|\le C\|r_N\|_{H^3_{\#}}
\]
and that for all $\varepsilon>0$,
\[
 \int_0^\pi|\varphi'_N(\theta)|\,{\rm
d}\theta\le C_\varepsilon \big(N^{7/2}\|r_N\|_{L^2(0,\pi)}
+N^{-\varepsilon}\|r_N\|_{H^{7/2+\varepsilon}_{\#}}
\big).
\]
Hence, for all $s'>7/2$  we deduce
\begin{equation}
\label{eq:quadRule:k2}
  | I_\kappa (F)- I_{\kappa,N} (F)|\le
  C_{s'} \kappa^{-2} \big(\|r_N\|_{H^{3}_{\#}}+N^{7/2}\|r_N\|_{H^{0}_{\#}}
+N^{7/2-s'}\|r_N\|_{H^{s'}_{\#}}\big).
\end{equation}

Using  \eqref{eq:quadRule:k0} and \eqref{eq:02:theo:main01} of
Theorem
\ref{theo:main01} with $s=0$ yields
\[
 \|r_N\|_{L^2_{\sin}}=\|r_N\|_{W_0^0}\le C_r N^{-r}\|F\|_{{\cal
H}^r}.
\]
Thus~\eqref{eq:02:theo:convQuadRule} holds for $\ell=0$.

For $\ell=1,2$, we need to bound $\|r_N\|_{H^s_{\#}}$ for different values of
$s$. 
Applying~\eqref{eq:02:lemma:quadRule}, with $g={\cal F}_0F$ and $g_N= {\rm
q}_N^{\rm
e}\rho_N^0F$  and   \eqref{eq:02:theo:main01} of Theorem \ref{theo:main01} 
we obtain for $r>s+1/2$
\begin{eqnarray}
 \|r_N\|_{H^s_{\#}}&\le&
 C_{r}\big[ N^{s-r+1/2}\|{\cal F}_0F\|_{W_0^r}+
N^{s+1/2} \|r_N\|_{W_0^0}+
N^{s-1/2}\| r_N \|_{W_0^1}\big]\nonumber\\
&\le& C_{r}\Big[ N^{s-r+1/2}\|{\cal F}_0F\|_{W_0^r}+
N^{s-r+1/2} \Big[\sum_{\ell=-\infty}^\infty \|{\cal
F}_{m+2\ell N }F\|^2_{W_{m+2\ell N }^{r}}\Big]^{1/2}\bigg]\nonumber\\
&\le& C_{r} \sqrt{2} N^{s-r+1/2}\|F\|_{{\cal H}^r}. \label{eq:quadRule:05}
\end{eqnarray}
Applying \eqref{eq:quadRule:05}  first in \eqref{eq:quadRule:k1} with $s=1$ and
second in \eqref{eq:quadRule:k2} with $s=3$, $s=0$ and
$s=s'<r-1/2$, we 
obtain the result~\eqref{eq:02:theo:convQuadRule} for $\ell = 1, 2$.
\end{proof}

\vspace{-0.6in}
\section{Numerical Experiments }\label{sec:num_exp}

In the section, we demonstrate the algorithms and analysis developed in this
article for interpolatory approximations of functions on the sphere 
with various smoothness properties and efficient evaluation of  
the wideband integrals in~\eqref{eq:integ} induced by such functions. 
For a fixed observation point $\x^* \in \sph$ and
a smoothness parameter $s$,   we consider 
$F^\circ_s: \sph \rightarrow \R$, defined by 
\begin{equation}\label{eq:ex_fn}
F^\circ_s(\x) = \left|\x - \x^*\right|^{s} = 
\left[(x_1 - x_1^*)^2 + (x_2 - x_2^*)^2 + (x_3 - x_3^*)^2\right]^{s/2},
\qquad \x \in \sph,
\end{equation}
where $\x =  [x_1, x_2, x_3]^T$ and $\x^* =  [x_1^*, x_2^*, x_3^*]^T$.

Clearly, for $s < 0$, $F^\circ_s$ is a discontinuous spherical function;
for $s \geq 0$ and even,   $F^\circ_s$ is a spherical polynomial;
and for $s \geq 1$ and odd $F^\circ_s\in {\cal H}^{r}(\sph)$, for
all $r < s+1$. 
For numerical experiments, 
we chose $\x^* =  [2/3, 1/3, 2/3]^T$; $s = 1, 3, 5$ and  various 
wideband wavenumbers  $\kappa = 10^{m}$ with $m = -2, -1, 0, 1, 2, 3, 4, 5$.

For $s = 1, 3, 5$,  Theorem~\ref{theo:main} indicates
that the estimated order of convergence (EOC) for
approximating $F_s$   by ${\cal Q}^{\rm gl}_N F_s$
in the ${\cal{H}}^0$ and ${\cal{H}}^1$ norms are approximately $s+1$ and $s$.
Results in  Tables~\ref{F1_int_tab}, \ref{F3_int_tab}, and  \ref{F5_int_tab}  
substantiate the efficient  approximations and analysis for  
spherical functions with various smoothness properties.  The
Sobolev norm errors were computed by approximating integrals in
\eqref{eq:normH0} and \eqref{eq:normH1} using rectangular rules with over 
$160, 000$ quadrature points.  

Further, for $s = 1, 3, 5$, using Theorem~\ref{theo:convQuadRule}, 
the error in approximating the wideband
integrals by the new class of cubature,
$\left| I_\kappa (F_s)- I_{\kappa,N}^{\rm gl} (F_s)\right|$, is similar to
that of $\|{\cal Q}^{\rm gl}_N F_s- F_s\|_{{\cal H}^0}$
for all wideband wavenumbers
and that the cubature error is much smaller for large wavenumbers.
Results in Tables~\ref{F1_quad_tab}, \ref{F3_quad_tab}, and  \ref{F5_quad_tab} 
demonstrate the very efficient approach developed in the article 
to approximate the wideband integrals and associated error analysis. 
For a fixed $N$ and for large $\kappa \geq 100$, results in 
Tables~\ref{F1_quad_tab}, \ref{F3_quad_tab}, and  \ref{F5_quad_tab} 
indicate ${\cal O}(\kappa^{-2})$ decay of the cubature error.  This is the
case even for $F_1$
that does not satisfy the smoothness assumption in
Theorem~\ref{theo:convQuadRule}.
In particular,   for a fixed large   $\kappa$, accurate solutions have been
obtained
even for small values $N$ and hence we do not observe consistent 
error improvement for certain increment in values of $N$ for the $F_1$ case.
Consequently, similar to~\cite[Table~2]{DoGrSm:2010} and as explained
in~\cite[Experiement~2]{DoGrSm:2010}, for  a fixed large $\kappa$,
the estimate in~Theorem~\ref{theo:convQuadRule} for $F_1$ is only an
upper bound for the error.

Using the algorithm described in Section~\ref{sec:theIntegrals}, in addition to
obtaining high-order accuracy,  the total MATLAB computation time taken to
produce all results in Tables~\ref{F1_quad_tab}, \ref{F3_quad_tab}, and 
\ref{F5_quad_tab} was less than  one-tenth of
one second of a single core of a 3.4GHz   INTEL CORE i7-2600 CPU. This is a
marked
computational advantage over (i) standard cubature approximations that
require at least ten points per wavelength to compute such wideband integrals
with low-order accuracy; and (ii) asymptotic based approximations and analysis
that are applicable only for very large wavenumbers. 

\clearpage

\begin{table}[h]
\begin{center}
\caption{\label{F1_int_tab} ${\cal H}^0$  and
${\cal H}^1$ errors  in 
approximation of $F_1$ by ${\cal Q}^{\rm gl}_NF_1$ and EOC} 
\begin{tabular}{|r|c|c|c|c|}
\hline
\text{$N$} &\text{ $\|{\cal Q}^{\rm gl}_N F_1- F_1\|_{{\cal
H}^0}$}&\text{EOC(${\cal H}^0)$}&\text{
$\|{\cal Q}^{\rm gl}_N F_1- F_1\|_{{\cal H}^1}$}&\text{EOC(${\cal H}^1)$}\\
\hline
5    & 4.0077e-02 &    & 3.3636{\rm E}-01 &   \\
10    & 1.1383e-02 &  1.8159  & 1.6166e-01 & 1.0570  \\
20   & 2.7244e-03 &2.0629   & 8.9315e-02 &  0.8560  \\
40   & 6.7527e-04 & 2.0124   & 3.9452e-02 &  1.1788   \\
80   & 1.6554e-04 & 2.0283  & 2.0598e-02 &0.9376   \\ 
\hline
\end{tabular}

\vspace{0.02in}
\caption{\label{F3_int_tab} ${\cal H}^0$  and
${\cal H}^1$ errors  in 
approximation of $F_3$ by ${\cal Q}^{\rm gl}_NF_3$ and EOC}
\begin{tabular}{|r|c|c|c|c|}
\hline
\text{$N$} &\text{ $\|{\cal Q}^{\rm gl}_N F_3 - F_3\|_{{\cal
H}^0}$}&\text{EOC(${\cal H}^0)$}&\text{
$\|{\cal Q}^{\rm gl}_N F_3- F_3\|_{{\cal H}^1}$}&\text{EOC(${\cal H}^1)$}\\
\hline
5    & 9.1450e-03 &    & 6.4589e-02 &   \\
10    &  5.7752e-04 & 3.9850   & 6.8649e-03 & 3.2340  \\
20   & 3.4725e-05 &4.0558   & 9.7395e-04 & 2.8173  \\
40   & 2.2636e-06 &3.9393   & 1.0621e-04 &3.1970   \\
80   & 1.4386e-07 & 3.9758  & 1.4652e-05 & 2.8577   \\ 
\hline
\end{tabular}

\vspace{0.02in}
\caption{\label{F5_int_tab} ${\cal H}^0$  and
${\cal H}^1$ errors  in 
approximation of $F_5$ by ${\cal Q}^{\rm gl}_NF_5$ and EOC}
\begin{tabular}{|r|c|c|c|c|}
\hline
\text{$N$} &\text{ $\|{\cal Q}^{\rm gl}_N F_5-  F_5\|_{{\cal
H}^0}$}&\text{EOC(${\cal H}^0)$}&\text{
$\|{\cal Q}^{\rm gl}_N F_5- F_5\|_{{\cal H}^1}$}&\text{EOC(${\cal H}^1)$}\\
\hline
5    & 8.7351e-03 &    &  5.9412{\rm E}-02 &   \\
10    & 1.0756e-04 &6.3436   & 1.2333e-03 & 5.5901  \\
20   & 1.5172e-06 &6.1476   & 4.0722e-05 & 4.9206   \\
40   & 2.5432e-08 &5.8986   & 1.1236e-06 &5.1796   \\
80   & 4.0461e-10 & 5.9740   & 3.9233e-08 &4.8399   \\ 
\hline
\end{tabular}

\vspace{0.02in}
\caption{\label{F1_quad_tab} 
Cubature errors $\left| I_\kappa (F_1)- I_{\kappa,N}^{\rm gl} (F_1)\right|$ 
for $\kappa = 10^m, ~m = -2, -1, 0, 1, 2, 3, 4, 5$}
\begin{tabular}{|r|c|c|c|c|c|c|c|c|}
\hline
\text{$N$} &\text{ $\kappa = 10^{-2}$}&\text{ $\kappa =
10^{-1}$}&\text{ $\kappa = 10^{0}$}
&\text{ $\kappa = 10^{1}$}&\text{ $\kappa = 10^{2}$}&\text{ $\kappa = 10^{3}$}
&\text{ $\kappa = 10^{4}$} &\text{ $\kappa = 10^{5}$} \\
\hline                                                                                                          5 & 2.0e-03  & 2.1e-03  & 3.2e-03  & 1.1e-02  & 7.3e-05 & 9.5e-07& 7.7e-09  &
7.4e-11 \\  
10 & 1.3e-03 & 1.3e-03  & 1.3e-03  & 3.6e-03  & 2.6e-05 & 1.8e-07& 1.4e-09  &
1.3e-11 \\ 
20 & 3.7e-05 & 3.7e-05  & 3.8e-05  & 1.3e-04  & 1.4e-05 & 1.4e-07& 1.8e-09  &
1.9e-11 \\ 
40 & 1.6e-05 & 1.6e-05  & 1.6e-05  & 1.7e-05  & 1.4e-05 & 8.1e-09& 9.7e-11  &
9.7e-13 \\ 
\hline
\end{tabular}

\vspace{0.02in}
\caption{\label{F3_quad_tab} Cubature errors 
$\left| I_\kappa (F_3)- I_{\kappa,N}^{\rm gl} (F_3)\right|$ 
for $\kappa = 10^m, ~m = -2, -1, 0, 1, 2, 3, 4, 5$}
\begin{tabular}{|r|c|c|c|c|c|c|c|c|}
\hline
\text{$N$} &\text{ $\kappa = 10^{-2}$}&\text{ $\kappa =
10^{-1}$}&\text{ $\kappa = 10^{0}$}
&\text{ $\kappa = 10^{1}$}&\text{ $\kappa = 10^{2}$}&\text{ $\kappa = 10^{3}$}
&\text{ $\kappa = 10^{4}$} &\text{ $\kappa = 10^{5}$} \\
\hline  
 5& 1.2e-04& 1.4e-04 & 3.9e-04 & 1.7e-03& 1.5e-05& 1.9e-07& 1.6e-09& 1.4e-11 \\
 
10& 3.0e-05 & 3.0e-05& 3.0e-05 & 2.3e-04& 1.7e-06& 2.1e-08& 1.7e-10& 1.6e-12 \\
 
20&1.7e-07& 1.7e-07& 2.0e-07&  1.2e-06& 3.1e-07&3.2e-09& 4.0e-11 & 4.1e-13 \\  
40& 2.3e-08& 2.3e-08& 2.3e-08& 2.8e-08 & 2.8e-08& 2.1e-11& 3.2e-13& 3.3e-15 \\  
\hline
\end{tabular}

\vspace{0.02in}
\caption{\label{F5_quad_tab} Cubature errors 
$\left| I_\kappa (F_5)- I_{\kappa,N}^{\rm gl} (F_5)\right|$ 
for $\kappa = 10^m, ~m = -2, -1, 0, 1, 2, 3, 4, 5$}
\begin{tabular}{|r|c|c|c|c|c|c|c|c|}
\hline
\text{$N$} &\text{ $\kappa = 10^{-2}$}&\text{ $\kappa =
10^{-1}$}&\text{ $\kappa = 10^{0}$}
&\text{ $\kappa = 10^{1}$}&\text{ $\kappa = 10^{2}$}&\text{ $\kappa = 10^{3}$}
&\text{ $\kappa = 10^{4}$} &\text{ $\kappa = 10^{5}$} \\
\hline                                                                                                   
5& 3.1e-05& 7.0e-05& 1.5e-04& 8.1e-04& 7.2e-06& 1.0e-07& 7.0e-10& 6.3e-12 \\  
10& 1.7e-06& 1.7e-06& 1.8e-06& 4.5e-05& 4.2e-07& 5.5e-09& 4.3e-11& 4.1e-13\\  
20&1.4e-09& 1.4e-09& 2.4e-09&2.9e-08& 1.9e-08& 1.9e-10& 2.3e-12& 2.3e-14\\ 
40& 9.0e-11& 9.0e-11& 9.0e-11& 1.3e-10&2.7e-10& 8.1e-13& 3.2e-15& 4.8e-18\\  
\hline
\end{tabular}

\end{center}
\end{table}

\clearpage

\appendix

\section{Proofs of results in Section~\ref{sec:prop_fourier_spaces}}
\label{sec:fourier_space_appendix}
\begin{proof} (Proposition~\ref{prop:eqWm1Y1}.) 
To establish this result we will use \eqref{eq:defNormWms} with 
\eqref{eq:normH1} and \eqref{eq:normH2}. Without loss of
generality, we  assume $f$ to be a real valued function.
The first identity~\eqref{eq:01:theo:eqWm1Y1}
follows readily by noticing that
\[
|\nabla_{\Sp}
(f\otimes e_m)(\theta,\phi)|^2 =\frac{1}{2\pi}\Big[ \frac{m^2}{\sin^2\theta}
 |f(\theta)|^2  +|f'(\theta)|^2\Big].
\]
The proof of \eqref{eq:02:theo:eqWm1Y1} requires more calculations and 
application of integration by parts several times to take care of some
cross products appearing in the integral form of the norm. Without
loss of generality,  for a fixed $m \in \Z$, we can assume $f\in {\rm
span}\left\{ Q_n^m\ :  n\ge |m|\right\}$ because this subspace
is dense in $W_m^2$.
Observe that
\begin{equation}\label{eq:f:cancellation}
f(0)=f'(0)=f(\pi)=f'(\pi)=0.
\end{equation}
Since  
\begin{eqnarray*}
\frac{1}4\|\nabla_\Sp
f\otimes e_{m}\|_{{\cal H}^0\times {\cal
H}^0}^2+\frac{1}{16}\|f\otimes e_{m}\|_{{\cal H}^0}^2&=&
\frac{m^2}{4}\int_0^\pi
f^2(\theta)\frac{{\rm d}\theta}{\sin\theta}+\frac{1}{4}\int_0^\pi
|f'(\theta)|^2{\sin\theta}  \,{\rm d}\theta\\
&&+
 \frac{1}{16}\int_0^\pi
f^2(\theta)\sin\theta\,{\rm d}\theta, 
\end{eqnarray*} it is sufficient  to analyze the  term containing the
Laplace-Beltrami operator:
\begin{eqnarray}
 \|\Delta_\Sp (f\otimes e_{m})\|_{{\cal H}^0}^2&=& \int_0^{\pi}
 \int_0^{2\pi}  |\Delta_{\Sp} (f\otimes e_m)(\theta,\phi)|^2\sin\theta\,{\rm
d}\theta\,{\rm d}\phi\nonumber\\
&=& 
\int_0^{\pi}
   \bigg|-\frac{m^2}{\sin^2\theta} f(\theta)+
\frac{1}{\sin\theta}\big|\big(\sin\theta
f'(\theta)\big)'\bigg|^2 \sin\theta\,{\rm
d} \theta\nonumber\\
&=&
 {m^4}\int_0^\pi f^2(\theta)\,\frac{{\rm d}\theta}{\sin^3\theta}+\int_0^\pi
 \frac{1}{\sin\theta}\big|\big(\sin\theta
f'(\theta)\big)'\big|^2\,{\rm d}\theta\nonumber\\
&& -
2m^2 \int_0^\pi f(\theta) \big(\sin\theta
f'(\theta)\big)'  \frac{{\rm
 d}\theta}{\sin^2\theta}  =:m^4 I_1+I_2-2m^2I_3.\
\quad\label{eq:03:theo:eqWm1Y1}
\end{eqnarray} 
Using~\eqref{eq:f:cancellation} and integration by parts, 
\begin{eqnarray}
 I_2&=&\int_0^\pi  \Big [ \frac{\cos^2\theta}{\sin\theta} |f'(\theta)|^2\
+\sin\theta|f''(\theta)|^2+\cos\theta
\big(|f'(\theta)|^2\big)'\Big]  \,{\rm d}\theta \nonumber\\
&=&
\int_0^\pi \frac{\cos^2\theta}{\sin\theta} |f'(\theta)|^2\,{\rm
d}\theta+\int_0^\pi \sin\theta|f''(\theta)|^2\,{\rm
d}\theta+\int_0^\pi
 |f'(\theta)|^2\sin\theta \,{\rm d}\theta\nonumber
\\&=&
\int_0^\pi  |f'(\theta)|^2\,\frac{{\rm
d}\theta}{\sin\theta} +\int_0^\pi \sin\theta|f''(\theta)|^2\,{\rm
d}\theta.  \label{eq:04:theo:eqWm1Y1}
\end{eqnarray}
 
Proceeding similarly, we  derive
\begin{eqnarray}
 I_3&=& -\int_0^\pi \Big(\frac{1}{\sin^2\theta} f(\theta)\Big)'
f'(\theta)\sin\theta\,{\rm d}\theta\nonumber\\
&=&-\int_0^\pi |f'(\theta)|^2\frac{\rm d\theta}{\sin\theta}+
\int_0^\pi \big(f^2(\theta)\big)'\frac{\cos\theta}{\sin^2\theta}\,{\rm
d}\theta\nonumber\\
&=&-\int_0^\pi |f'(\theta)|^2\frac{\rm d\theta}{\sin\theta}+
\int_0^\pi f^2(\theta)
\big(\frac{2\cos^2\theta}{\sin^3\theta}+\frac{1}{\sin\theta}\big)\,{\rm
d}\theta\nonumber\\
&=&-\int_0^\pi |f'(\theta)|^2\frac{\rm d\theta}{\sin\theta}- \int_0^\pi
f^2(\theta) \frac{{\rm
d}\theta}{\sin\theta} +
 2\int_0^\pi f^2(\theta) \frac{{\rm
d}\theta}{\sin^3\theta}\,{\rm
d}\theta. \label{eq:05:theo:eqWm1Y1}
\end{eqnarray}
Inserting \eqref{eq:04:theo:eqWm1Y1}-\eqref{eq:05:theo:eqWm1Y1} in
\eqref{eq:03:theo:eqWm1Y1}, we obtain the desired result.
\end{proof}

\begin{proof} (Corollary~\ref{cor:equivNorms}.)
 For $|m|\ge 1$, \eqref{eq:01:cor:equivNorms} follows from
\eqref{eq:01:theo:eqWm1Y1}.

The inequalities
\[
\frac{1}{16}+\frac{9m^2}{4}+(m^4-4m^2)\le m^4< 3m^4,\quad \frac14+(1+2m^2)\le
3
m^2,\qquad \forall |m|\ge 2
\]
with \eqref{eq:02:theo:eqWm1Y1}  imply the first
inequality of \eqref{eq:02:cor:equivNorms}.
For $|m|\ge 3$ the second inequality of \eqref{eq:02:cor:equivNorms} is simply
a consequence of the inequalities 
\[
m^4-4m^2> m^4/2> \frac{m^4}{6},\quad 1+2m^2> m^2\ge
\frac{m^2}{6}.
\]
The case
$|m|=2$, has to be analyzed 
separately since one of the crucial terms, the third term in
\eqref{eq:02:theo:eqWm1Y1},   vanishes: Using \eqref{eq:02:theo:eqWm1Y1},
we obtain 
\begin{equation}\label{eq:03.5:cor:equivNormns}
\|f\|_{W_{m}^2}^2\ge 9\int_0^\pi |f(\theta)|^2\,\frac{\rm
d\theta}{\sin\theta}+9\int_0^\pi |f'(\theta)|^2\,\frac{\rm d\theta}{\sin\theta}+
\int_0^\pi |f''(\theta)|^2{\sin\theta}\, {\rm d\theta}. 
\end{equation}
 As before it suffices  to consider $f$ to be real valued and
that  $f\in {\rm span}\left\{ Q_n^2 : n \geq 2 \right\}$.
Note that
\begin{equation}\label{eq:04:cor:equivNormns}
 \int_0^\pi f^2(\theta)\frac{\rm d\theta}{\sin^3\theta}=
 \int_0^\pi f^2(\theta)\frac{\rm d\theta}{\sin\theta}+ \int_0^\pi
f^2(\theta)\frac{\cos^2\theta}{\sin^3\theta}{\rm d\theta}.
\end{equation}
Applying integration by parts to the second term and 
using~\eqref{eq:f:cancellation} 
we obtain
\begin{eqnarray} \label{eq:int_parts}
 \int_0^\pi
f^2(\theta)\frac{\cos^2\theta}{\sin^3\theta}{\rm d\theta}&=&
  \int_0^\pi \big(f(\theta)f'(\theta)\big)
\Big(\log(\tan(\theta/2))+\frac{\cos \theta}{\sin^2\theta}\Big)\,{\rm
d}\theta.
\end{eqnarray}
Notice that for $\theta\in(0,\pi/2]$,
\begin{equation}\label{eq:bound:log}
\sin\theta |\log \tan(\theta/2)|\le 2\tan(\theta/2) |\log \tan(\theta/2) |
 \le 2 e^{-1}\le 1.
\end{equation}
By symmetry, we can extend this bound for any $\theta\in(0,\pi)$.
With the help of \eqref{eq:bound:log} and the inequality  $2ab\le a^2+b^2$, 
from~\eqref{eq:int_parts}  we obtain
\begin{eqnarray}
 \int_0^\pi
f^2(\theta)\frac{\cos^2\theta}{\sin^3\theta}{\rm d\theta}&\le &
  \int_0^\pi   f(\theta)f'(\theta)  \,\frac{{\rm
d}\theta}{\sin\theta}+   \int_0^\pi   f(\theta)f'(\theta) \frac{{\rm
d}\theta}{\sin^2\theta}\nonumber\\
&\le& \frac12\bigg[\int_0^\pi   f^2 (\theta)  \,\frac{{\rm
d}\theta}{\sin\theta} +2\int_0^\pi  |f'(\theta)|^2  \frac{{\rm
d}\theta}{\sin\theta}\bigg] +     \frac{1}2\int_0^\pi   f^2(\theta) \frac{{\rm
d}\theta}{\sin^3\theta}. \qquad\label{eq:05:cor:equivNormns}
\end{eqnarray}
Inserting \eqref{eq:05:cor:equivNormns} in \eqref{eq:04:cor:equivNormns} we
easily  derive
\[
 \int_0^\pi
f^2(\theta)\frac{\rm d\theta}{\sin^3\theta}\le \frac{3}2\int_0^\pi  
f^2(\theta) \,\frac{{\rm
d}\theta}{\sin\theta}+\int_0^\pi  
|f'(\theta)|^2 \,\frac{{\rm
d}\theta}{\sin\theta}+\frac{1}2\int_0^\pi 
f^2(\theta)\frac{\rm d\theta}{\sin^3\theta}
\]
and therefore
\begin{eqnarray}
 \int_0^\pi f^2(\theta)\frac{\rm d\theta}{\sin^3\theta}&\le& 3
\int_0^\pi
f^2(\theta)\frac{\rm d\theta}{\sin\theta}+2\int_0^\pi
|f'(\theta)|^2\frac{{\rm
d}\theta}{\sin\theta}.\label{eq:06:cor:equivNormns}
\end{eqnarray}
From \eqref{eq:03.5:cor:equivNormns} and
\eqref{eq:06:cor:equivNormns}, we   obtain 
\begin{eqnarray*}
 6\|f\|^2_{W_{m}^2}&\ge& 54\int_0^\pi f^2(\theta)\,\frac{\rm
d\theta}{\sin\theta}+54\int_0^\pi |f'(\theta)|^2\,\frac{\rm
d\theta}{\sin\theta}+6
\int_0^\pi |f''(\theta)|^2{\sin\theta}\, {\rm d\theta}\\
&\ge&16\bigg( 3 \int_0^\pi
f^2(\theta)\frac{\rm d\theta}{\sin\theta}+2\int_0^\pi
|f'(\theta)|^2\frac{{\rm
d}\theta}{\sin\theta}\bigg)\\
&&+4\int_0^\pi
|f'(\theta)|^2\frac{{\rm
d}\theta}{\sin\theta}+
\int_0^\pi |f''(\theta)|^2{\sin\theta}\, {\rm d\theta}\\
&\ge&16\int_0^\pi
|f (\theta)|^2\frac{{\rm
d}\theta}{\sin^3\theta}+4\int_0^\pi
|f'(\theta)|^2\frac{{\rm
d}\theta}{\sin\theta}+
\int_0^\pi |f''(\theta)|^2{\sin\theta}\, {\rm d\theta}.
\end{eqnarray*}

Hence  the inequalities in~\eqref{eq:02:cor:equivNorms} hold.
\end{proof}

\begin{proof} (Proposition~\ref{prop:inclusion_in_Hr}.)
Denote by $\Gamma$ the maximum circle in $\Sp$, parametrized by
 \begin{equation}
 \label{eq:chi}
\q(\theta):=(\sin\theta, 0 ,\cos\theta).
 \end{equation}
Given $f^\circ:\Gamma\to \mathbb{C}$   we denote   $f=f^\circ\circ 
\q:\R\to\mathbb{C}$.   The norm  in the Sobolev space  $H^r(\Gamma)$  can 
be then
defined with the help of $\q$ and~\eqref{eq:SobolevNorm}: 
\[
\|f^\circ\|_{H^r(\Gamma)}:=\|f\|_{H_{\#}^r}.
\]

The second ingredient we will use in this proof is the trace
operator $\gamma_\Gamma$ which can be shown to be  continuous from ${\cal
H}^{r+1/2}(\Sp)$ onto $H^{r}(\Gamma)$ for all
$r>0$ (see \cite{MR937473, McLean:2000} for a proof of this result in $\R^n$;
the proof can be easily extended by using local charts of the unit sphere and
the equivalent definitions of the Sobolev spaces involved).

Given $f\in W_m^r$, consider the mapping
\[
 {\cal P}_m f:=\sqrt{2\pi}(\gamma_\Gamma F^\circ )\circ \q, \qquad
F^\circ:=(f\otimes e_m) \circ \p^{-1}.
\]
Observe that $F^\circ\in {\cal H}^r(\Sp)$ and that actually $f= {\cal P}_m f$,
that is ${\cal P}_m$ is simply the identity operator. 
Moreover,
\[
\| {\cal P}_m f\|_{H_\#^r} \le
{\sqrt{2\pi}}\|\gamma_\Gamma\|_{{\cal H}^{r+1/2}(\Sp)\to
H^r(\Gamma)}\|F^\circ\|_{{\cal
H}^{r+1/2}(\Sp)}={\sqrt{2\pi}}\|\gamma_\Gamma\|_{{\cal H}^{r+1/2}(\Sp)\to
H^r(\Gamma)} \|f\|_{W_m^{r+1/2}},
\]
where $\|\gamma_\Gamma\|_{{\cal H}^{r+1/2}(\Sp)\to
H^r(\Gamma)}$ is the continuity constant of $\gamma_\Gamma$ as a linear  
operator from
${\cal  H}^{r+1/2}(\Sp)$ onto ${H^{r}(\Gamma)}$. 
Hence we obtain \eqref{eq:01:prop:inclusion_in_Hr}.

Since $W_m^0\cong
L^2_{\sin}$, \eqref{eq:03a:prop:inclusion_in_Hr} is
clear whereas  \eqref{eq:03b:prop:inclusion_in_Hr}  for $m=0$ follows directly
from
 \eqref{eq:01:theo:eqWm1Y1} and \eqref{eq:SobolevNorm:2}.

Finally, if $m\ne 0$   using $f(0)=0$ (see Remark
\ref{remark:continuityWm}), we observe that
\begin{eqnarray*}
 \int_0^{\pi/2}|f(\theta)|^2\frac{\rm d\theta}{\rm \sin\theta}&=&
 \int_0^{\pi/2}\frac1{\sin \theta}\bigg|\int_0^\theta f'(\xi)\,{\rm d}\xi
\bigg|^2 {\rm d \theta}\le \int_0^{\pi/2} \frac{\sqrt{\theta}}{\sin
\theta}\bigg[
\int_0^\theta |f'(\xi)|^2\,{\rm d}\xi\bigg]\,{\rm d}\theta\\
&\le& C \int_0^{\pi/2} |f'(\xi)|^2\,{\rm
d}\xi.
\end{eqnarray*} 
Proceeding similarly, but using now that
$f(\pi)=0$,
we can bound   the integral in $(\pi/2,\pi)$ and  hence  conclude that 
\[
  \int_0^{\pi}|f(\theta)|^2\frac{\rm d\theta}{\rm \sin\theta}\le C
\int_0^{\pi} |f'(\xi)|^2\,{\rm
d}\xi. 
\]
Equation \eqref{eq:01:cor:equivNorms}  now yields that
\[
 \|f\|_{W_m^1}\le\frac{\sqrt{5}}2\|f\|_{{Z_m^1}}\le C(1+|m|)\|f\|_{H_\#^1}. 
\]
\end{proof}

\begin{proof} (Proposition~\ref{prop:ineqY1}.)
For the sake of a simpler notation, we will assume without loss of generality
that $f$ is  a  real valued function. We consider the functions
\begin{eqnarray*}
 \varphi_1(\theta)&:=&\left\{
\begin{array}{ll}
1,\quad&\text{if }\theta< \pi/6,\\
 \displaystyle 2-\frac{6\theta}{\pi} \displaystyle,\quad&\text{if
}\theta\in[\pi/6,\pi/3],\\
0,\quad&\text{otherwise},
\end{array}
\right.\\
\varphi_3&:=&\varphi_1(\pi-\,\cdot\,),\qquad   \varphi_2\:
:=\:1-\varphi_1-\varphi_3.
\end{eqnarray*}
Obviously,  $0\le \varphi_j\le 1$ and actually
$\{\varphi_1,\varphi_2,\varphi_3\}\subset H^1(0,\pi)$ form a partition of
unity.
Thus
\begin{equation}\label{eq:int_pou}
 \int_0^\pi  f ^2(\theta)  \,{\rm d}\theta=\sum_{j=1}^3 \int_0^\pi
 f ^2(\theta) \varphi_j(\theta)\,{\rm d}\theta=:\sum_{j=1}^3 I_j.
\end{equation}
Clearly
\begin{eqnarray}
\label{eq01:lemma:ineqY1}
I_2&=&\int_{\pi/6}^{5\pi/6}   f ^2(\theta)\varphi_2(\theta)\,{\rm
d}\theta\le 2\int_{\pi/6}^{5\pi/6} f ^2(\theta)\sin \theta\,{\rm
d}\theta.
\end{eqnarray}
On the other hand
\begin{eqnarray*}
I_1&=&\int_0^{\pi/3}  f ^2(\theta) \varphi_1(\theta)\,{\rm d}\theta
=  f ^2(\theta)\varphi_1(\theta)  \theta \:\Big|^{\theta=\pi/3}_{\theta=0}-
\int_0^{\pi/3}  \big(f^2(\theta)\varphi_1(\theta)\big)'\theta\, {\rm d}\theta.
\end{eqnarray*}
The first term is zero, since $\varphi_1(\pi/3)=0$. For  the other
term we first note that
\[
\|\varphi'_1\|_{L^\infty(0,\pi)}\le \frac{6}{\pi},\qquad
\frac{\theta}{\sin\theta}\le
\frac{2\pi}{3\sqrt{3}}\approx 1.209,\quad \forall\theta\in[0,\pi/3].
\]
Thus,
\begin{eqnarray}
\bigg|
\int_0^{\pi/3}  \big(f^2(\theta)\varphi_1(\theta)\big)'\theta\, {\rm
d}\theta\bigg| &=&
\bigg|\int_0^{\pi/3}
\big(2f(\theta)f'(\theta)+f^2(\theta)\varphi'_1(\theta)\big)\frac{\theta}{
\sin\theta}\,\sin\theta\, {\rm d}\theta\bigg| \nonumber \\
&&\ \hspace{-2.9cm}\le \frac{2\pi}{3\sqrt{3}}\bigg[2\int_0^{\pi/3} |f(\theta)
f'(\theta)| \sin \theta \,{\rm
d}\theta+ \int_0^{\pi/3} f^2(\theta) |\varphi_1'(\theta)|\,\sin\theta\, {\rm
d}\theta\bigg]\nonumber\\
&&\ \hspace{-2.9cm}\le \frac{4}{\sqrt{3}}
\bigg[2\int_0^{\pi/3} |f(\theta)
f'(\theta)| \sin \theta \,{\rm
d}\theta+ \int_0^{\pi/3} f^2(\theta) \,\sin\theta\, {\rm
d}\theta\bigg]\nonumber\\
&&\ \hspace{-2.9cm}\le \frac{4}{\sqrt{3}} \bigg[
\bigg(\int_0^{\pi/3} \!\!f^2(\theta)\,\sin\theta\,{\rm
d}\theta\bigg)^{1/2} \! \bigg(\int_0^{\pi/3}\!\!
\big|f'(\theta)\big|^2\,\sin\theta\,{\rm
d}\theta\bigg)^{1/2}\nonumber\\
&&\hspace{-1.6cm}+
\int_0^{\pi/3} f^2(\theta)\,\sin\theta\,{\rm
d}\theta\bigg].
\label{eq02:lemma:ineqY1}
\end{eqnarray}

Proceeding in a similar way, we can prove
\begin{equation}
I_3 \le \frac{4}{\sqrt{3}} \bigg[
\bigg(\int_{2\pi/3}^\pi \!\!f^2(\theta)\,\sin\theta\,{\rm
d}\theta\bigg)^{1/2} \! \bigg(\int_{2\pi/3}^\pi \!\!
\big|f'(\theta)\big|^2\,\sin\theta\,{\rm
d}\theta\bigg)^{1/2}\!\!+\!
\int_{2\pi/3}^\pi f^2(\theta)\,\sin\theta\,{\rm
d}\theta\bigg].  \label{eq03:lemma:ineqY1}
\end{equation}
 
Substituting
(\ref{eq01:lemma:ineqY1})-(\ref{eq03:lemma:ineqY1}) in~\eqref{eq:int_pou}, we
obtain 
the desired result. 
\end{proof}

Recall that if $\Omega$ is open and $\Gamma\subset \Omega$ is a Lipschitz
curve, then~\cite[Theorem 1.6.6]{BrSc:2002}
\[
\|\gamma_\Gamma u\|^2_{L^2(\Gamma)}\le
C_\Omega \|u\|_{L^2(\Omega)}\|u\|_{H^1(\Omega)}.
\]
It is easy to see that Proposition \ref{prop:ineqY1} can be
alternatively derived from this result
by taking $\Omega=\Sp$, $\Gamma$  the circle parameterized by \eqref{eq:chi} and
$u=f\circ e_m$.

\section{Proofs of Lemma~\ref{lemma:pL2} and 
Proposition~\ref{prop:convqNeo}}
\label{sec:appB_fourier_anal_proof}

In this section we prove several auxiliary results,  most concerned   with
convergence estimates  of  the interpolants ${\rm q}_{N,\rm
gl}^{\rm m}$, introduced in \eqref{eq:qGL}--\eqref{eq:qNm}, 
in the spaces $W_m^r$,  which are, up to our knowledge, new or sharper than that
found in the literature (see for instance \cite{BeMa:1992} or \cite[\S
6.6]{Fun:1992} and references therein).

The first result is an inverse inequality for $\mathbb{D}_N^{\rm e}$ and
$\mathbb{D}_{N}^{\rm o}$ which holds in $W_m^s$, uniformly in $m$ and $N$.

\begin{lemma} \label{lemma:inverseIneq}
There exists
$C >0$ such that for any $m\in \mathbb{Z}$,    $g_N\in \mathbb{D}_N^{\rm e}\cup
\mathbb{D}_N^{\rm o}$ with $g_N(0)=g_N(\pi)=0$ if $m\ne 0$, 
\[
  \|g_N\|_{W_m^1}\le  C (1+|m|)N  \|g_N\|_{W_m^0}.
\]
\end{lemma}
\begin{proof}
We first prove  the result for even $m$. Let $g_N\in {\mathbb D}_N^{\rm e}$
and take $m'=2$ if $g_N(0)=g_N(\pi)=0$ and
$m'=0$ otherwise. Note that 
\[
g_N\in \mathop{\rm span}\big\{ Q_n^{m'}\ :\ n=m',\ldots, N \big\} , 
\]
(this property is not valid for $m'>2$). Since
$\{Q_n^{m'}\}_{n=m'}^\infty$ is an orthonormal basis of $L^2_{\rm
sin}=W_{m}^0$, we obtain
\[
 g_N=\sum_{n=m'}^N  \bigg[\int_0^\pi g_N(\theta) Q_n^{m'}(\theta)\sin
\theta\,{\rm d}\theta\bigg] Q_n^{m'}=\sum_{n=m'}^N  (\widehat{g_N})_{m'}(n)
 Q_n^{m'}.
\]
Hence
\[
 \|g_N\|^2_{W_{m'}^1}= \sum_{n=m'}^N (n+{\textstyle\frac12})^2
|  (\widehat{g_N})_{m' }(n)|^2 \le
 (N+{\textstyle\frac12})^2\sum_{n=m'}^N
| 
(\widehat{g_N})_{m'}(n)|^2=(N+{\textstyle\frac12})^2\|g_N\|^2_{W_{m'}^0}.
\]
This concludes the proof for  $|m|=0,2$. For even $|m|>2$, we
apply
\eqref{eq:01:theo:eqWm1Y1} of Proposition \ref{prop:eqWm1Y1} to derive
\begin{eqnarray*}
 \|g_N\|_{W_{m}^1}&\le& \frac{ |m|}{2}
\|g_N\|_{W_2^1}\le  \frac{ |m|}{2} 
(N+{\textstyle\frac12}
) \|g_N\|_{L^2_{\sin}}=\frac{ |m|}{2} 
(N+{\textstyle\frac12}
) \|g_N\|_{W_m^0}.
\end{eqnarray*}

The proof for odd $m$ is similar:  Since
$\mathbb{D}_{N}^{\rm
o}= {\rm
span}\left\{ Q_n^1\ :  n=1,\ldots, N\right\}$, as above the
result can be derived first for $m = 1$ and 
using  again  \eqref{eq:01:theo:eqWm1Y1} of
Proposition~\ref{prop:eqWm1Y1}, 
the result follows for all  odd values of $m$. 
\end{proof}

Let
\begin{equation}
\label{eq:defrhoNm}
 {\rm s}_N^m f:=\sum_{|m|\le n\le N-1} \widehat{f}_m(n) Q_n^m, \qquad
 \widehat{f}_m(n)=\int_{0}^\pi f(\theta)Q_n^{m}(\theta)\sin\theta\,{\rm
d}\theta,
\end{equation}
with the convention that, for $|m|\ge N$, ${\rm s}_N^m f=0$. Clearly
${\rm s}_N^m$ is just the orthogonal projection onto ${\rm
span}\left\{ Q_n^m\ :\  n=|m|,\ldots, N-1\right\}$ in  $W_m^s$ for all $s$.
Moreover $\mathrm{s}_{N}^{2\ell}f\in \mathbb{D}_{N-1}^{\rm
e}\subset \mathbb{D}_{N}^{\rm e}$, $\mathrm{s}_{N}^{2\ell+1}f\in
\mathbb{D}_{N-2}^{\rm o}$,  and therefore
\begin{equation}
   \label{eq:idqNpN}
 {\rm q}_N^{\rm e} s_{N}^{2\ell}= s_{N}^{2\ell},\qquad 
 {\rm q}_N^{\rm o} s_{N}^{2\ell+1}= s_{N}^{2\ell+1}.
\end{equation}

\begin{lemma}\label{lemma:convRhoN}
For all $r\ge s$ and $N\ge 2$,
\begin{equation}
\label{eq:convRhoN}
 \|f-{\rm s}_N^mf\|_{W_m^s}\le
(N+{\textstyle\frac12})^{s-r}\|f\|_{W_m^r}.
\end{equation}
Moreover, if $m\ne 0$,
\begin{equation}\label{eq:sec_zero}
 {\rm s}_N^mf(0)={\rm s}_N^mf(\pi)=0,
\end{equation}
and
\begin{equation}
\label{eq02:convRhoN}\max\big\{
|{\rm s}_N^0f(0)-f(0)|,
|{\rm s}_N^0f(\pi)-f(\pi)|\big\}\le
\frac{N^{1-r}}{(2r-2)^{1/2}}\|f\|_{W_0^r},
\end{equation}
for all $r>1$.
\end{lemma}
\begin{proof}
Take $N'=\max\{N,|m|\}$. Then
\begin{eqnarray*}
 \|{\rm s}_N^m f-f\|^2_{W_m^s}& =&\sum_{n\ge N'}
\big(n+{\textstyle\frac12}\big)^{2s}|\widehat{f}_m(n)|^2\le
(N'+{\textstyle\frac12})^{2s-2r}\sum_{n\ge N'}
\big(n+{\textstyle\frac12}\big)^{2r}|\widehat{f}_m(n)|^2\\
&\le&
\big(N+{\textstyle\frac12}\big)^{2s-2r}\|f\|^2_{W_m^r},
\end{eqnarray*}
which proves the first result.
Using  \eqref{eq:Legendres}--\eqref{eq:qLeg},
\[
 Q_n^{m}(0)=Q_n^{m}(\pi)=0,\qquad \text{for  $m\ne 0$},
\]
and hence~\eqref{eq:sec_zero} follows.

For $m=0$ we note first that
\[
 Q_n^0(0)=\big(n+ {\textstyle\frac12}\big)^{1/2}
=(-1)^nQ_n^0(\pi).
\]
Then
\begin{eqnarray*}
 |{\rm s}_N^0 f(0)-f(0)|&\le& \sum_{n\ge N} |\widehat{f}_0(n)| 
 Q_n^0(0)
= \sum_{n\ge N } \big(n+
{\textstyle\frac12}\big)^{1/2}|\widehat{f}_0(n)|\\
&\le& \bigg[\sum_{n\ge N}
\big(n+{\textstyle\frac12}\big)^{-(2r-1)}\bigg]^{1/2}
\bigg[\sum_{n\ge N} \big(n+{\textstyle\frac12}\big)^{2r}
|\widehat{f}_0(n)|^2\bigg]^{1/2}\\
&\le&
\frac{{N}^{1-r}}{(2r-2)^{1/2}}\| f\|_{W_0^r}
\end{eqnarray*}
where, for $\alpha>1$, we have applied  
\[
\sum_{n\ge N} \big(n+{\textstyle\frac12}\big)^{-\alpha}\le
\int_{N-1/2}^\infty
\big(x+{\textstyle\frac12}\big)^{-\alpha}\,{\rm d}x
\le \frac{1}{\alpha-1}N^{1-\alpha}.
\]
Proceeding in a similar way, we can bound $ |{\rm s}_N^0 f(\pi)-f(\pi)|$ and
hence 
we obtain~\eqref{eq02:convRhoN}.
\end{proof}

 Our next result gives us information about the distribution of the nodes
$\{\theta_j\}_{j=0}^N$, the nodes of  ${\cal Q}_N^{\rm gl}$ in the
elevation angle. Roughly speaking, this distribution turns out to be
quasi-uniform.

\begin{lemma}\label{lemma:ajbj}
For $N>0$ define
\[
\begin{array}{rclrcll}
 a_j&:=& \frac{(j-1/4)\pi}{N+1/2}\qquad&  \displaystyle b_j&:=&
\frac{(j+1)\pi}{N+1/2},\qquad& j=1,\ldots, \lfloor\frac{N-1}{2}\rfloor,\\
 a_{ N /2}&:=&  \frac{\pi}2-\frac{5\pi/8}{N+1/2}\qquad&
 \displaystyle b_{ N /2}&:=&
\frac{\pi}2+\frac{5\pi/8}{N+1/2},\qquad &\mbox{for even $N$ },\\
b_j&:=&\pi-a_{N-j}\qquad&
a_j&:=&\pi-b_{N-j},\qquad&
j= \lceil\frac{N+1}{2}\rceil,\ldots,N-1.\\
\end{array}
\]
Let $\{\theta_j\}_{j=0}^N$ be as in \eqref{eq:deftheta}. Then
\[
\theta_j\in[a_j,b_j],\qquad j=1,\ldots,N-1.
\]
\end{lemma}
\begin{proof} Note that for even $N$,
$\theta_{N/2}=\arccos \xi_{N/2} =\arccos 0=\pi/2$ and therefore there is
nothing to prove in this case. On the other hand,
\[
\arccos \zeta_{N-j+1}  <\theta_j=\arccos\eta_{N-j}<
\arccos \zeta_{N-j}
\]
where $\{\zeta_j\}_{j=1}^N$ are the zeros of  $P_N$, the
Legendre polynomial of degree
$N$ (see \eqref{eq:Legendres}). We can apply well-known results on the
asymptotic distribution of these roots (see for instance
\cite[(3.1.2)-(3.1.3)]{Fun:1992})  to deduce that for $1\le j<N/2$,
\[
 a_j=\frac{(j-1/4)\pi}{N+1/2} <
\arccos \zeta_{N-j+1} <\theta_j<\arccos \zeta_{N-j} <\frac{(j+1)\pi}{N+1/2}=
b_j.
\]
(See also \cite[Theorem 6.21.3]{szego} or \cite[Theorem 2.4]{BeMa:1992}.)
For $j>N/2$ the result follows from the first part of the proof and the
symmetric disposition of the nodes.
\end{proof}

Let us introduce the Gauss-Lobatto quadrature rule
\[
 \sum_{j=0}^N \omega_j g(\eta_j)\approx \int_{-1}^{1} g(s)\,{\rm d}s, 
\]
which is exact for any polynomial of degree up to
$2N-1$~\cite{DaRa:1975,KrUe:1998}. 
We point out that the weights of this formula satisfy (see for instance
\cite[Section 3.5]{Fun:1992}, \cite[Corollary 2.8]{BeMa:1992})
\[
\omega_0=\omega_{N}=\frac{2}{ N (N+1)}
\]
and
\begin{equation}
\label{eq:omegaj}
 0< \omega_j\le  \frac{C}{N}\sin\theta_j,\qquad j=1,\ldots,N-1
\end{equation}
with $C$ independent of $N$. 
 
Consequently,  the quadrature
rule 
\[
 {\cal L}_N f:=\sum_{j=0}^{N}\omega_j f(\theta_j)\approx \int_{0}^\pi
f(\theta)\,\sin\theta\,{\rm d}\theta =\int_{-1}^1 f(\arccos s)\,{\rm d}s   
\]
is exact for any  $p_N\in \mathbb{D}_{2N-1}^{\rm e}$ 
(recall that
$\theta_{j}=\arccos\eta_{N-j}$). Hence,
\[
 {\cal L}_N(|p_N|^2)=\|p_N\|^2_{L^2_{\sin}}, \quad \forall
p_N\in\mathbb{D}_{N-1}^{\rm
e}\cup\mathbb{D}_{N-2}^{\rm o}. 
\]
If  $p_N\in\mathbb{D}_{N}^{\rm
e}\cup\mathbb{D}_{N-1}^{\rm o}$ (see \eqref{eq:def:DNeDNo}), the rule fails to
be exact for $|p_N|^2\in \mathbb{D}_{2N}^{\rm e} $. However,  we have
instead the bound
\begin{equation}
 \label{eq:boundRule}
 \int_0^\pi  |p_N(\theta)|^2\sin\theta\,{\rm d}\theta\le  C\sum_{n=0}^N \omega_n
|p_N(\theta_n)|^2,\qquad p_N\in\mathbb{D}_{N}^{\rm
e}\cup\mathbb{D}_{N-2}^{\rm o}
\end{equation}
with $C$ independent of $N$ and $p_N$.
This bound, which is a consequence of a well known  stability result for
Gauss-Lobatto
formulas (see \cite[Theorem 3.8.2]{Fun:1992}), will play an essential role in
what follows.

 We are ready to prove the first inequality in  Lemma \ref{lemma:pL2}. 
In fact below we  prove  a slightly more general result (see Corollary
\ref{cor:equivNorms}
which gives an equivalent norm for $W_m^1$).

\begin{lemma}\label{lemma:pL2prime}
There exists $C>0$ such that for all  $N\ge 2$  and for any $f\in{\cal
C}^0[0,\pi]$ 
with $ f'\in
L^2_{\sin}$   
\begin{eqnarray*}
\|{\rm q}_{N,\rm gl}^{m} f\|_{L^2_{\sin}}&\le& C\Big[
\|f\|_{L^2_{\sin}} + \frac{1}{N} \|f'\|_{L^2_{\sin}} +
\frac{1}{N}\big(| f(0)|+|  f(\pi)|\big)\Big],
\end{eqnarray*}
with $C$ is independent of $f$ and $N$.
\end{lemma}
\begin{proof} Without loss of  generality  we  assume that  $f$
is a real valued function.  Recall that from the definition of
${\rm q}_{N,\rm gl}^{m}$, 
\begin{eqnarray*}
 {\rm q}_{N,\rm gl}^{2\ell} f(\theta_j)&=&{\rm q}_{N,\rm gl}^{{\rm e}}
f(\theta_j)=f(\theta_j),\quad j=0,\ldots, N,\\
 {\rm q}_{N,\rm gl}^{2\ell+1} f(\theta_j)&=&{\rm q}_{N,\rm gl}^{{\rm o}}
f(\theta_j)=f(\theta_j),\quad j=1,\ldots, N-1,\quad  {\rm q}_{N,\rm
gl}^{2\ell+1} f(0)={\rm q}_{N,\rm
gl}^{2\ell+1} f(\pi)=0. 
\end{eqnarray*}
Then,  using \eqref{eq:boundRule} we obtain
\[
\int_0^\pi  [{\rm q}_{N,\rm gl}^{m}f]^2(\theta)\sin\theta\,{\rm d}\theta\le
C_m
\sum_{n=0}^N \omega_n
[{\rm q}_{N,\rm gl}^{m}f]^2(\theta_n) \le C_m\sum_{n=0}^{N}
\omega_n
f^2(\theta_n),
\]
where $C_m$ is independent of $f$ and $N$.
 Note that $C_m =1$ for $m$ odd (that is, for  ${\rm q}_{N,\rm
gl}^{\rm o}f$) or for $m$ even with  ${\rm q}_{N,\rm
gl}^{\rm e}f \in \mathbb{D}^{\rm e}_{N-1}$. 
From Lemma~\ref{lemma:ajbj} we have
\begin{equation}
\theta_n\in[a_n,b_n],\quad \text{with}\quad 
 \label{eq:an-bn}
 \frac{\pi}{N}\le \frac{5\pi}{4(N+1/2)}=(b_n-a_n)\le \frac{2\pi}{N},\quad
n=1,\ldots,N-1.
\end{equation} 
Hence from~\eqref{eq:omegaj} and~\eqref{eq:an-bn} we conclude that
\begin{eqnarray}
\int_0^\pi [{\rm q}_{N,\rm
gl}^{m}f]^2 (\theta)\sin\theta\,{\rm d}\theta&\le&
 \frac{2 C_m}{N(N+1)}
\big(f^2(0)+f^2(\pi)\big)+C\bigg[\frac{1}{N}\sum_{n=1}^{N-1}
f^2(\theta_n)\sin\theta_n\bigg]\nonumber \\
&\le&\frac{2C_m}{N(N+1)}
\big(f^2(0)+f^2(\pi)\big)+\frac{C}{\pi}\sum_{n=1}^{N-1} (b_n-a_n)
f^2(\theta_n)\sin\theta_n.\nonumber \\ \label{eq:01:GL:InteriorNodes}
\end{eqnarray}

We point out that the above sum  can be understood as a
composite rectangular rule
applied to $f^2(\theta)\sin\theta$. 
 Using the well-known bound (with $c\in[a,b]$),
\[
\bigg|
 \int_a^b g(\theta)\,{\rm d}\theta-g(c)(b-a)\bigg|\le
(b-a)\int_{a}^b|g'(\theta)|\,{\rm d}\theta,
\]
and  noticing that 
\[
 \sum_{n=1}^{N-1} 
\int_{a_n}^{b_n} \big|g(\theta)\big|\,{\rm d}\theta\le
2\int_0^\pi|g(\theta)|\,{\rm d}\theta,\quad \forall g\in L^1(0,\pi),
\]
we
deduce
\begin{eqnarray*}
\sum_{n=1}^{N-1}   (b_n-a_n)  f^2(\theta_n)\sin\theta_n&&\\
&&\hspace{-3.5 cm}=
\sum_{n=1}^{N-1} \bigg[ (b_n-a_n)  f^2(\theta_n)\sin\theta_n -\int_{a_n}^{b_n}
f^2(\theta)\sin\theta  \, {\rm d}\theta\bigg]+\sum_{n=1}^{N-1}\int_{a_n}^{b_n}
f^2(\theta)\sin\theta \, {\rm d}\theta\\
&& \hspace{-3.5cm} \le \sum_{n=1}^{N-1}(b_n-a_n)
\int_{a_n}^{b_n} \big|\big(f^2(\theta)\sin\theta\big)'\big|\,{\rm d}\theta
+2\int_{0}^\pi
f^2(\theta)\sin\theta\,  {\rm d}\theta\\
&& \hspace{-3.5cm}\le\frac{4\pi}{N}\bigg[\int_{0}^\pi
f^2(\theta)\,
|\cos\theta|\,{\rm d}\theta+\int_{0}^\pi 2|f(\theta)f'(\theta)| \sin\theta\,{\rm
d}\theta
\bigg]+2\int_{0}^\pi f^2(\theta)\sin\theta \,  {\rm d}\theta\\
&&\hspace{-3.5 cm}\le\frac{C'}{N}\bigg[\int_{0}^\pi f^2(\theta)
\,{\rm d}\theta+\bigg(\int_{0}^\pi f^2(\theta)\sin\theta\,  {\rm
d}\theta\bigg)^{1/2} \bigg(\int_{0}^\pi |f'(\theta)|^2\sin\theta\,  {\rm
d}\theta\bigg)^{1/2}
\bigg]\\
&&\hspace{-2.5cm}+2\int_{0}^\pi
f^2(\theta)\sin\theta\,  {\rm d}\theta\\
&&\hspace{-3.5cm}\le C'' \int_{0}^\pi \!\!f^2(\theta)
 \sin\theta\, {\rm d}\theta
+\frac{C''}{N}\! \bigg(\int_{0}^\pi\!\!
f^2(\theta)\sin\theta\, {\rm
d}\theta\bigg)^{1/2}\! \bigg(\int_{0}^\pi |f'(\theta)|^2\sin\theta
\, {\rm
d}\theta\bigg)^{1/2}
\end{eqnarray*}
where we have applied again  \eqref{eq:an-bn} and used  in the
last step Proposition \ref{prop:ineqY1}. Finally,
the
inequality
\[
 2ab\le N a^2+ N ^{-1}  b ^2
\]
proves that 
\begin{equation}
 \label{eq:02:GL:InteriorNodes}
 \sum_{n=1}^{N-1}   (b_n-a_n)  f^2(\theta_n)\sin\theta_n\le
C\bigg[\int_{0}^\pi
f^2(\theta)\sin\theta\, {\rm
d}\theta +\frac{1}{ N ^2} \int_{0}^\pi |f'(\theta)|^2\sin\theta\,
{\rm
d}\theta
\bigg]. 
\end{equation}
Substituting \eqref{eq:02:GL:InteriorNodes} in
\eqref{eq:01:GL:InteriorNodes},  the desired result follows.
\end{proof}

\begin{remark}
Inequalities such as those proved in   Lemma \ref{lemma:pL2prime} have been
used in
the literature for studying the convergence of polynomial interpolants in
weighted Sobolev spaces (see for instance \cite{Fun:1992} or \cite{BeMa:1992}
and references therein).

Note that with the change of variables $s=\cos\theta$,  Lemma
\ref{lemma:pL2} implies that for  $f\in{\cal C}^1[-1,1]$ with $f(\pm 1)=0$,
and  if $p_N$  is the polynomial of degree $N$ which interpolates $f$ at
Gauss-Lobatto nodes $\{\eta_j\}_{j=0}^N$, we have the bound
\[
 \int_{-1}^1 |p_N(s)|^2\,{\rm d}s\le C\bigg(
 \int_{-1}^1 |f(s)|^2\,{\rm d}s+ \frac{1}{N^2}\int_{-1}^1 |f'(s)|^2(1-s^2)\,{\rm
d}s
\bigg)
\]
with $C$ independent of $N$ and $f$.
This bound is  an improvement of   Theorem 4.1 in \cite{BeMa:1992}
where a similar result is proven but without the weight function $(1-s^2)$  in
the second integral.
\end{remark} Next we prove Proposition~\ref{prop:convqNeo} for the $s=0$ case. 
\begin{proposition}\label{prop:conv01}
For  $r> 1$ there exists $C_r>0$ so that
\[
\|{\rm q}_{N,\rm gl}^{m}f-f\|_{W_m^0}\le C_r N^{-r}\|f\|_{W_m^r}.
\]
Moreover, the estimate above holds also for $r=1$ if $m\ne 0$. 
\end{proposition}
\begin{proof} Assume  first  that
$m\ne 0$. \ We recall from  Remark~\ref{remark:continuityWm}  that any $f\in
W_m^1$ is  
continuous and vanishes at $\{0,\pi\}$. 
From  \eqref{eq:idqNpN},  Lemma
\ref{lemma:pL2prime} (see also Lemma \ref{lemma:pL2}), 
Proposition \ref{prop:eqWm1Y1}, and  \eqref{eq:convRhoN}
of Lemma \ref{lemma:convRhoN}  we deduce that for 
$r\ge 1$,
\begin{eqnarray*}
 \|{\rm q}_{N,\rm gl}^{m} f-f\|_{W_m^0}&=&
 \|{\rm q}_{N,\rm gl}^{m} 
(f-{\rm s}_N^mf)-(f-{\rm s}_N^mf)\|_{W_m^0}\\
&\le& C \big[\|f-{\rm s}_N^m f\|_{W_m^0}+N^{-1}\|f-{\rm
s}_N^mf\|_{W_m^1}\big]\le
CN^{-r} \|f\|_{W_m^r}.
\end{eqnarray*}

If $m=0$, we have instead
\[
  \|{\rm q}_{N,\rm gl}^{\rm e} f-f\|_{W_0^0}\le C
\Big(\|f-{\rm s}_N^0f\|_{W_0^0}+N^{-1}\Big[\|f-{\rm s}_N^0f\|_{W_0^1} +
 |(f-{\rm s}_N^0f)(0)|+
 |(f-{\rm s}_N^0f)(\pi)|\Big]\Big).
\]
Using \eqref{eq:convRhoN} as before, and  \eqref{eq02:convRhoN}  for
taking care of the pointwise errors at $\{0,\pi\}$ (and here using
the restriction $r>1$) we conclude that
\[
  \|{\rm q}_{N,\rm gl}^{\rm e} f-f\|_{W_0^0}\le C N^{-r}\|f\|_{W_0^r}.
\]
\end{proof}

To derive convergence estimates in $W_m^1$ \ we require additional 
technical results. To this end, it is convenient to  introduce  the norms
\begin{eqnarray*}
 \|f\|_{L^2_{\sin^{k}}}&:=& \bigg[\int_0^\pi
|f(\theta)|^2\,{\sin^k\theta}\, {\rm d \theta}\bigg]^{1/2}, 
\end{eqnarray*}
and the associated spaces $L^2_{{\sin^k}}$.  Below we will  use  these norms for
 $k=-3,-1$ and $2$ .

\begin{lemma} \label{lemma:aux:qn}
For all $N\ge 2$ and for all $p_N\in\mathbb{D}_N^{\rm e}\cup
\mathbb{D}_{N-1}^{\rm
o}$,
\[
\|p_N\|_{L^2(0,\pi)}\le 
\Big[\frac{1}{\sin({\textstyle\frac{\pi}{2N+4}})}\Big]^{1/2}
\|p_{N}\|_{L^2_{\sin}}.
\]
\end{lemma}
\begin{proof} In this proof  we will apply the fact that the
quadrature 
rule
\[
 \frac{\pi}{N+2}\sum_{j=0}^{N+1}
f \big({\textstyle\frac{(j+1/2)\pi}{N+2}}\big)\approx
\int_0^{\pi}f(\theta)\,{\rm d}\theta
\]
is exact for any $f\in\mathbb{D}_{2N+2}^{\rm e}$. Notice that, by hypothesis,
$|p_N|^2\in\mathbb{D}_{2N}^{\rm e}$. Hence, 
\begin{eqnarray}
 \|p_N\|_{L^2(0,\pi)}^2\!\!\!\!&=&\!\!\frac{\pi}{N+2}\sum_{j=0}^{N+1}
\big|p_N\big({\textstyle\frac{(j+1/2)\pi}{N+2}}\big)\big|^2\nonumber\\ 
&\le&\!\!\!\!\Big[
\sin\big({\textstyle\frac{\pi}{2N+4}}\big)\Big]^{-2}\!\!\!\frac{\pi}{N+2}\sum_{
j=0 } ^ { N+1}
\big|(p_N\sin)\big({\textstyle\frac{(j+1/2)\pi}{N+2}}\big)\big|^2
=\big[
\sin\big({\textstyle\frac{\pi}{2N+4}}\big)\big]^{-2} \|p_N\|^2_{
L^2_{\sin^2}}.\nonumber \\ 
\label{eq:IneqNorm}
\end{eqnarray}

We point out that, in the usual notation of the theory of
interpolation of Hilbert spaces, we have
\[
 L^2_{\sin}=\big[L^2(0,\pi), L^2_{\sin^2} \big]_{1/2}
\]
(see for instance \cite[Lemma 23.1]{Tar:2007}). Hence, with 
\eqref{eq:IneqNorm}, we obtain the desired result:
\[
 \|p_N\|_{ L^2(0,\pi)}^2\le \big[
\sin\big({\textstyle\frac{\pi}{2N+4}}\big)\big]^{-1}
 \|p_N\|^2_{ L^2_{\sin}},\qquad \forall p_N\in\mathbb{D}^{\rm e}_{N}\cup
\mathbb{D}^{\rm o}_{N-1}.
\]
\end{proof}

\begin{lemma}\label{lemma:pL2prime2}
There exists $C>0$ so that for all $f\in W_m^2$ with $|m|\ge 2$, 
\begin{equation}\label{eq:LemB8_bd}
 \| {\rm q}_{N,\rm gl}^{m}f\|_{L^2_{\sin^{-1}}}\le 
C\Big[\|f\|_{L^2_{\sin^{-1}}}+
\frac{1}{N}
\big(\|f\|_{L^2_{\sin^{-3}}}+\|f\|_{L^2_{\sin^{-1}}}
\big) \Big].
\end{equation}
 \end{lemma}
\begin{proof} First, consider the case of even $m$  so that 
${\rm q}_{N,\rm gl}^m f = {\rm q}_{N,\rm gl}^{\rm e} f$.
Then, it is easy
to
check that
\[
({\rm q}_{N,\rm gl}^{\rm e}
f)(\theta)/\sin\theta ={\rm q}_{N,\rm gl}^{\rm o}\big(f/\sin\big)(\theta).
\]
(Note that in this case $f(0)=f(\pi)=0$.)
We can proceed as in the proof of Lemma \ref{lemma:pL2prime}, see
\eqref{eq:01:GL:InteriorNodes}--\eqref{eq:02:GL:InteriorNodes}, and derive
\begin{eqnarray*}
  \| {\rm q}_{N,\rm gl}^{\rm e}f\|^2_{L^2_{\sin^{-1}}}&=&
  \| {\rm q}_{N,\rm gl}^{\rm o}(f/\sin)\|^2_{L^2_{\sin}}\le
C\sum_{n=1}^{N-1}  
(b_n-a_n)  \big(f/\sin\big)^2(\theta_n)\sin\theta_n\\
& \le&
C'\bigg[\int_{0}^\pi
f^2(\theta)\frac{\rm
d\theta}{\sin\theta} +\frac{1}{N^2} \int_{0}^\pi
\big|\big({f/\sin}\big)'(\theta)\big|^2\sin\theta\,
{\rm
d}\theta
\bigg]\\
&\le&  
2C'\bigg[\int_{0}^\pi
f^2(\theta)\frac{\rm
d\theta}{\sin\theta} +\frac{1}{N^2} \int_{0}^\pi
 f^2(\theta) \frac{\rm
d\theta}{\sin^3\theta}+\frac{1}{N^2}\int_{0}^\pi
|f' (\theta)|^2 \frac{\rm
d\theta}{\sin\theta}\bigg].
\end{eqnarray*}
This proves the result for the even $m$ case.

To prove the result for the odd $m$ case,  observe that
now ${\rm q}_{N,\rm
gl}^{\rm o}f /\sin \in\mathbb{D}_{N-2}^{\rm e}$. However, if  we  try to apply
the
same ideas as before, unlike the
previous case, we have to take care also of the pointwise values at $0$ and
$\pi$. Indeed, following again
\eqref{eq:01:GL:InteriorNodes}--\eqref{eq:02:GL:InteriorNodes},
and noticing that $C_m=1$ in \eqref{eq:01:GL:InteriorNodes}, 
we derive
\begin{eqnarray}
\| {\rm q}_{N,\rm gl}^{ \rm
o }f\|^2_{L^2_{\sin^{-1}}}\!\!\!\!&=&\|{\rm q}_{N,\rm
gl}^{\rm o}f/\sin\|^2_{L^2_{\sin}} \le  \frac{2}{N(N+1)}\big(
|\big({\rm q}_{N,\rm gl}^{\rm o}
f /\sin\big)(0)|^2+|\big({\rm q}_{N,\rm gl}^{\rm o}
f /\sin\big)(\pi)|^2
\big)\nonumber\\
&&+C
\sum_{n=1}^{N-1} 
(b_n-a_n)  \big(f/\sin\big)^2(\theta_n)\sin\theta_n\nonumber\\
&\le& 
\frac{2}{N(N+1)}\big(| \big({\rm q}_{N,\rm gl}^{\rm o}
f /\sin\big)(0)|^2+|\big({\rm q}_{N,\rm gl}^{\rm o}
f /\sin\big)(\pi)|^2
\big) \nonumber\\
&&\!\!+
2C'\bigg[\int_{0}^\pi
f^2(\theta)\frac{\rm d\theta}{\sin\theta} +\frac{1}{N^2} \int_{0}^\pi
 f^2(\theta) \frac{\rm
d\theta}{\sin^3\theta}+\frac{1}{N^2}\int_{0}^\pi
|f' (\theta)|^2 \frac{\rm
d\theta}{\sin\theta}\bigg]. \nonumber \\ \label{eq:ineq01}
\end{eqnarray}
It remains to bound 
\begin{eqnarray*}
R_N(f)&:=&
\frac{2}{N(N+1)}\big(
|\big({\rm q}_{N,\rm gl}^{\rm o}
f /\sin\big)(0)|^2+|\big({\rm q}_{N,\rm gl}^{\rm o}
f /\sin\big)(\pi)|^2
\big)\\
&\le& \frac{4}{N(N+1)} \|{\rm q}_{N,\rm gl}^{\rm o}
f/\sin\|^2_{L^\infty(0,\pi)}.
\end{eqnarray*}
We proceed as follows: Assume first that $f$ is {\em even}, that is 
\[
 f(\pi-\theta)=f(\theta). 
\]
Then, so is ${\rm q}_{N,\rm gl}^{\rm o}f$ so that
\[
({\rm q}_{N,\rm gl}^{\rm o}f/\sin)(\theta)=\sum_{n=0}^{\lfloor
(N-2)/2\rfloor}\alpha_{2n}\cos
2n\theta 
\]
for suitable coefficients $\alpha_{2n}$. Hence, by applying  the  Cauchy-Schwarz
inequality we obtain 
\begin{eqnarray*}
 \|{\rm q}_{N,\rm gl}^{\rm o}
f/\sin\|^2_{L^\infty(0,\pi)} &\le&   \bigg[\sum_{n=0}^{\lfloor
(N-2)/2\rfloor}|\alpha_{2n}|\bigg]^2\\
&\le&
\frac{2}{\pi}\Big[\pi |\alpha_{0}|^2+\sum_{n=1}^{\lfloor
(N-2)/2\rfloor}|\alpha_{2n}|^2{ \frac{\pi}2}\Big]\times 
\frac{N}2=\frac{N}{\pi}
\|{\rm q}_{N,\rm gl}^{\rm o}
f /\sin \|^2_{L^2(0,\pi)}. 
\end{eqnarray*}
Recall that  ${\rm q}_{N,\rm gl}^{\rm
o} f/\sin\in\mathbb{D}_{N-2}^{\rm e}\subset \mathbb{D}_{N-1}^{\rm e}$. 
Then we can apply Lemma
\ref{lemma:aux:qn} to get
\[
  \|{\rm q}_{N,\rm gl}^{\rm o}
f/\sin\|^2_{L^\infty(0,\pi)} \le \frac{N}{  \pi \sin
(\pi/((2N+2))} \|
{\rm q}_{N,\rm gl}^{\rm
o}
f/\sin\|^2_{L^2_{\sin}}, 
\]
and therefore
\[
R_N(f)\le  \frac{4}{\pi (N+1) \sin (\pi/(2N+2))}
\|{\rm q}_{N,\rm gl}^{\rm
o}
f/\sin
\|^2_{L^2_{\sin}}=:c(N) \|{\rm
q}_{N,\rm gl}^{\rm o}
f\|^2_{L^2_{\sin^{-1}}}. 
\]
We point out that
\[ 
    0<c(N)\le c(2)=\frac{8}{3\pi}<1,\qquad \forall N\ge 2.
\]
Using this bound in \eqref{eq:ineq01}, we obtain 
\begin{equation}
\|{\rm q}_{N,\rm gl}^{\rm o}
f\|^2_{L^2_{\sin^{-1}}}\le  \frac{2C}{1-c(2)} \bigg[
\int_{0}^\pi
f^2(\theta)\frac{\rm d\theta}{\sin\theta} +\frac{1}{N^2}\bigg(
\int_{0}^\pi
 f^2(\theta) \frac{\rm
d\theta}{\sin^3\theta}+\int_{0}^\pi
|f' (\theta)|^2 \frac{\rm
d\theta}{\sin\theta}\bigg)\bigg].\label{eq:aux:bound}  
\end{equation}
For {\em odd} $f$,  that is  for functions with 
$f(\theta)=-f(\pi-\theta)$, we can proceed similarly and show
that a similar bound to \eqref{eq:aux:bound} also holds  in
this case.

We have  proven the result for the subspace of {\em even} and {\em odd}
functions. To extend the result for arbitrary $f$ we note first that
\[
 f=f_{\rm e}+f_{\rm o}, \quad {\rm q}_{N,\rm gl}^{\rm o} f={\rm
q}_{N,\rm gl}^{\rm o} f_{\rm
e}+{\rm q}_{N,\rm gl}^{\rm o} f_{\rm o}
\]
where  $f_{\rm e}(\theta):=\frac{1}2(f(\theta)+f(\theta-\pi))$ and
$f_{\rm o}(\theta)=\frac{1}2(f(\theta)-f(\theta-\pi))$ are  {\em even}
and {\em odd} parts of $f$. We deduce that 
\begin{eqnarray*}
\|{\rm q}_{N,\rm gl}^{\rm o} f \|^2_{L^2_{\sin^{-1}}} &=&  
\|{\rm
q}_{N,\rm gl}^{\rm o}
f_{\rm e} \|^2_{L^2_{\sin^{-1}}}+   \|{\rm q}_{N,\rm gl}^{\rm o}
f_{\rm o} \|^2_{L^2_{\sin^{-1}}}\\
&  \le&
C\Big[\|f_{\rm e}\|^2_{L^2_{\sin^{-1}}}+\|f_{\rm o}\|^2_{L^2_{\sin^{-1}}}
\\& &\ 
\frac{1}{N^2}
\big(\|f_{\rm
e} \|^2_{L^2_{\sin^{-3}}}+\|f_{\rm o} \|^2_{L^2_{\sin^{-3}}} 
+\|(f_{\rm e})'\|^2_{L^2_{\sin^{-1}}}+\|(f_{\rm
o})'\|^2_{L^2_{\sin^{-1}}}
 \big)\Big]\\
& =&
C\Big[\|f\|^2_{L^2_{\sin^{-1}}} +
\frac{1}{N^2} 
\big(\|f\|^2_{L^2_{\sin^{-3}}}+\|f'\|^2_{L^2_{\sin^{-1}}}
 \big)\Big].
\end{eqnarray*} Thus the result~\eqref{eq:LemB8_bd}  follows.  
\end{proof}
We are now ready to prove  the second inequality in Lemma
\ref{lemma:pL2prime}. 

\begin{corollary}
\label{cor:pL2prime2} For all $|m|\ge 2$, 
\[
 \|{\rm q}_{N,\rm gl}^{m} f  \|_{{Z_m^1}}\le
C\Big[\|f\|_{{Z_m^1}}+\frac{1}{N}\|f\|_{{Z_m^2}}\Big].
\]
\end{corollary}
\begin{proof}
 Note that 
\begin{eqnarray*}
 \|{\rm q}_{N,\rm gl}^{m} f \|^2_{{Z_m^1}}&\le& m^2\|{\rm
q}^{m}_{N,\rm gl} f
\|_{L^2_{\sin^{-1}}}^2 +
 2\big\|\big({\rm q}_{N,\rm gl}^{m} (f-{\rm
s}_{N}^m f)\big)'\big\|_{L^2_{\sin}}^2+
 2\big\|  ({\rm s}_{N}^mf) '\big\|_{L^2_{\sin}}^2\\
&=:&I_1+I_2+I_3.
\end{eqnarray*} We bound  the first term using~Lemma~\ref{lemma:pL2prime2}: 
\begin{eqnarray*}
 I_1&\le& C
m^2\Big[\|f\|_{L^2_{\sin^{-1}}}^2+\frac{1}{N^2}
\Big(\|f\|_{L^2_{\sin^{-3}}}^2
+\|f'\|_{L^2_{\sin}}^2\Big)\Big]\le
 C\big[\|f\|_{{Z_m^1}}^2+\frac{1}{N^2} \|f\|_{Z_m^2}^2\big].
\end{eqnarray*}
For the second term,   using first  Proposition
\ref{prop:eqWm1Y1} and combining  the inverse inequalities stated in Lemma
\ref{lemma:inverseIneq} with Lemma \ref{lemma:convRhoN} and  Proposition
\ref{prop:conv01},   we obtain
\begin{eqnarray*}
 I_2&\le&  2\|
{\rm q}_{N,\rm gl}^{m} (f-{\rm
s}_{N}^m f)\big\|_{W_2^1}^2 \le C N^2\| {\rm q}_{N,\rm gl}^{m}
(f-{\rm
s}_{N}^m f)\big\|_{W_m^0}^2 \le C'  \|f\|_{W_m^1}^2\le {\textstyle \frac{
5}4}C'\|f\|^2_{Z_m^1}, 
\end{eqnarray*} 
where  in the last step we have
used~\eqref{eq:01:cor:equivNorms} of Corollary~\ref{cor:equivNorms}.
We bound the last term as
\begin{eqnarray*}
 I_3&\le& 2\big\|   {\rm s}_{N}^mf \big\|_{W_m^1}^2\le
2\|f\|_{W_m^1}^2 \le
\frac{5}{2}\|f\|_{{Z_m^1}}^2.
\end{eqnarray*}
Thus we obtain the desired result.
\end{proof}
 
Finally, we prove Proposition~\ref{prop:convqNeo} for the $s=1$ case.
\begin{proposition}\label{prop:conv02}
For all $r\ge 2$ there exists $C_r>0$ so that
\[
\|{\rm q}_{N,\rm gl}^{m}f-f\|_{W_m^1}\le C_r
N^{1-r}\|f\|_{W_m^r}.
\]
\end{proposition}
\begin{proof}  Note that 
\begin{eqnarray*}
 \|{\rm q}_{N,\rm gl}^{m}  f-f\|_{W_m^1}&=&
\|{\rm q}_{N,\rm gl}^{m} 
(f-{\rm s}_N^mf)-(f-{\rm s}_N^mf)\|_{W_m^1}\\
&\le&  \|{\rm q}_{N,\rm gl}^{m}
(f-{\rm s}_N^mf)\| _{W_m^1} + 
N^{1-r} \|f\|_{W_m^r}
\end{eqnarray*}
where we have applied \eqref{eq:idqNpN} and  \eqref{eq:convRhoN}
of Lemma \ref{lemma:convRhoN}. 

It remains to bound the first term. For $|m|\le 1$, we make use of 
Lemma \ref{lemma:inverseIneq} and Proposition \ref{prop:conv01} to obtain,
\begin{eqnarray*}
 \|{\rm q}_{N,\rm gl}^{m}
(f-{\rm s}_N^mf)\| _{W_m^1}
&\le& C N \|{\rm q}_{N,\rm gl}^{m}(f-{\rm s}_N^m f)\|_{W_m^0}
 \le  C   N\big[ \|{\rm q}_{N,\rm gl}^{m} f-f\|_{W_m^0}+
\|f-{\rm s}_N^m f \|_{W_m^0}\big]\\
& \le &
C'N^{1-r} \|f\|_{W_m^r}.
\end{eqnarray*} 
For $|m|\ge 2$, this approach is not valid since the inverse inequalities of
Lemma \ref{lemma:inverseIneq} contains $m$ as a penalizing term. Thus, we follow
a different approach  and use  Corollary
\ref{cor:equivNorms} and
\ref{cor:pL2prime2} to obtain
\begin{eqnarray*}
  \|{\rm q}_{N,\rm gl}^{m} (f-{\rm s}_N^mf)
\|_{W_m^1}\!\!&\le&\!\!
{\textstyle\frac{\sqrt{5}}2}\|{\rm q}_{N,\rm gl}^{m} (f-{\rm
s}_N^mf)\|_{{Z_m^1}}
\le C
\Big(\|f-{\rm s}_N^mf\|_{{Z_m^1}}+N^{-1}\|f-{\rm s}_N^mf\|_{{Z_m^2}}\Big)
 \\
&\le& \!\! C 
\Big(\|f-{\rm s}_N^mf\|_{W_m^1}+ \sqrt{6} N^{-1} \|f-{\rm
s}_N^mf\|_{W_m^2}\Big)\le C' N^{1-r}\|f\|_{W_m^r}.
\end{eqnarray*}
Hence the desired result follows.
\end{proof}

\paragraph{Acknowledgments.} The first author is supported by  Project
MTM2010-21037. Part of this research was carried out during a short visit of
the second author to Universidad P\'ublica de Navarra.
We thank Drs. Hawkins, Mhaskar, and Sayas  for several useful
discussions. 


\end{document}